\documentclass[review]{elsarticle}
\usepackage[breaklinks,
           pdfstartview=FitH,
           colorlinks=true,
           citecolor=blue]{hyperref}
\usepackage{amsfonts}
\usepackage{amsmath}
\usepackage{amsthm}
\usepackage{amssymb}
\usepackage{amscd,multirow}
\numberwithin{equation}{section}
\newtheorem{definition}{Definition}[section]
\newtheorem{theorem}{Theorem}[section]
\newtheorem{corollary}{Corollary}[section]
\newtheorem{lemma}{Lemma}[section]
\newtheorem{proposition}{Proposition}[section]
\newtheorem{remark}{Remark}[section]

\begin{document}

\begin{frontmatter}

\title{Asymptotic behavior of a nonautonomous evolution equation governed by a quasi-nonexpansive operator\tnoteref{mytitlenote}}
\tnotetext[mytitlenote]{This work was supported by  the National Natural Science  Foundation of China
(11471230) and the Scientific Research Foundation  of the Education Department of Sichuan Province (16ZA0213).}

\author[mymainaddress]{Ming Zhu}
\ead{mingzhu129@foxmail.com}

\author[mysecondaryaddress]{Rong Hu}
\ead{ronghumath@aliyun.com}

\author[mymainaddress]{Ya-Ping Fang\corref{mycorrespondingauthor}}
\cortext[mycorrespondingauthor]{Corresponding author}
\ead{ypfang@scu.edu.cn}

\address[mymainaddress]{Department of Mathematics, Sichuan University, Chengdu, Sichuan, P.R. China}
\address[mysecondaryaddress]{Department of Applied Mathematics, Chengdu University of Information Technology, Chengdu, Sichuan, P.R. China}

\begin{abstract}
We study the asymptotic behavior of the trajectory of a nonautonomous evolution equation governed by a quasi-nonexpansive operator in Hilbert spaces. We prove the weak convergence of the trajectory to a fixed  point of the operator by relying on Lyapunov analysis. Under a metric subregularity condition, we further derive a flexible global exponential-type rate for the distance of the trajectory to the set of fixed points. The results obtained are applied to analyze the  asymptotic behavior of the trajectory of an adaptive Douglas-Rachford dynamical system, which is applied for finding a zero of the sum of two operators, one of which is strongly monotone while the other one is weakly monotone.
\end{abstract}


\end{frontmatter}

\section{Introduction}
\label{intro}
Throughout this paper, $\mathbb{H}$ is a real Hilbert space endowed with  scalar $\langle\cdot,\cdot\rangle$ and
norm $\|\cdot\|$, $\mathbb{R}$ is the set of real numbers, $\mathbb{R}_{+}:=\{x\in\mathbb{R}\mid x\geq0\}$, and
$\mathbb{R}_{++}:=\{x\in\mathbb{R}\mid x>0\}$.
We utilize the notation $A:\mathbb{H}\rightrightarrows\mathbb{H}$ to indicate that $A$ is a set-valued operator on $\mathbb{H}$ and the notation $A:\mathbb{H}\rightarrow\mathbb{H}$ to indicate that $A$ is a single-valued on $\mathbb{H}$.
 Given two set-valued  operators $A$ and $B$: $\mathbb{H} \rightrightarrows \mathbb{H}$, the prototypical  inclusion problem
\begin{equation}
   \label{SMP}
   \text{find}\quad u\in \mathbb{H}\quad\text{such that}\quad  0\in A(u)+B(u),
\end{equation}
models a variety of tasks in diverse applied fields
such as signal processing, machine learning and statistics
\cite{application2Bruckstein}.
A popular and powerful method for solving the problem (\ref{SMP})
is the Douglas-Rachford (DR) algorithm
\cite{b9Douglas,b11Lions}.
In its formulation the operator is decomposed into simpler individuals which
are then processed separately in the subproblems, hence, DR algorithm is often referred  as a splitting algorithm.
In 1956, this splitting algorithm was introduced originally by Douglas and Rachford \cite{b9Douglas}
to solve heat conduction
flow problems in a finite dimensional space. The original splitting scheme is
$$\left\{
\begin{array}{l}
{\frac{1}{\lambda}\left(u_{k+\frac{1}{2}}-u_{k}\right)+A\left(u_{k+\frac{1}{2}}\right)+B\left(u_{k}\right)=0} \\ {\frac{1}{\lambda}\left(u_{k+1}-u_{k+\frac{1}{2}}\right)+B\left(u_{k+1}\right)-B\left(u_{k}\right)=0,}
\end{array}
\right.
$$
where $A$ and $B$
are single-valued linear monotone operators.
Eliminating $u_{k+\frac{1}{2}}$, and defining $z_{k}=\left(J_{\lambda B}\right)^{-1} u_{k}$,
one can rewrite the above scheme as
\begin{equation}\label{original-DPA}
z_{k+1}=J_{\lambda A}\left(2 J_{\lambda B}-I\right) z_{k}+\left(I-J_{\lambda B}\right) z_{k},
\end{equation}
where
$J_{\lambda A}$ and $J_{\lambda B}$  are the resolvent operators of $A$ and $B$ respectively, and $I$ is the identity operator in $\mathbb{H}$.
The algorithm (\ref{original-DPA})  is just the classical DR algorithm model.
In 1979, Lions and Mercier \cite{b11Lions} made the algorithm applicable to the problem
with $A$ and $B$ being set-valued nonlinear operators.
They proved in a Hilbert space that the algorithm converges
weakly to a point which can be used to solve the problem.
Later, Svaiter \cite{Svaiter2011} revealed that the shadow sequence associated with DR algorithm is weakly convergent to a solution.
In 1992, Eckstein and Bertsekas \cite{b10Eckstein} further analyzed the DR algorithm with summable errors as well as with over/under relaxation.
Moreover,
interpretation of DR algorithm as a
proximal point method and an alternating direction method of multipliers (ADMM)
can be dated back to \cite{Lawrence1987} and \cite{Gabay1983}, respectively.

Compared with the rich literature for DR algorithms with the involved operators
being monotone or strongly monotone (see e.g. \cite{Dao-Phan2018jg,b12-1He,HeB2012,Phan2016}), the convergence theory for weakly monotone settings is far from
being complete. When $A$ and $B$  are the  subdifferentials of a strongly convex function and a weakly convex function respectively,
some nice works on convergence analysis for DR algorithms can be found in \cite{b12Bayram,K-Guo2018,c29Guo}.
Recently, Dao and Phan \cite{Dao-Phan2018} proposed the following adaptive DR algorithm for the
problem (\ref{SMP}):
\begin{equation}
    \label{iterative2}
    z_{k+1}=\tilde{T}z_k:=(1-\epsilon)z_k+\epsilon R_{\delta B}^{\mu} R_{\gamma A}^{\lambda}z_k,
\end{equation}
where $\epsilon \in (0,1)$, $\lambda, \mu, \gamma, \delta>0$ and
\begin{equation*}
\begin{aligned}
J_{\gamma A}=(\operatorname{I}+\gamma A)^{-1}, \quad& R_{\gamma A}^{\lambda}=(1-\lambda) \mathrm{I}+\lambda J_{\gamma A},
\\ J_{\delta B}=(\mathrm{I}+\delta B)^{-1}, \quad& R_{\delta B}^{\mu}=(1-\mu) \mathrm{I}+\mu J_{\delta B}.
\end{aligned}
\end{equation*}
As a more general scheme than (\ref{original-DPA}), the algorithm (\ref{iterative2}) has a nice flexibility thanks to its adjustable parameters.
This plays an important role in its convergence analysis.
With proper tuning the parameters, Dao and Phan \cite{Dao-Phan2018} proved the convergence of the sequence generated by  (\ref{iterative2}) under the conditions that $A$ is strongly monotone
while $B$ is weakly monotone. They further demonstrated that the algorithm enjoys a global
linear convergence rate under the additional condition that $A$ or $B$ is Lipschitz continuous. Dao and Phan's results generalized and improved several contemporary works such as \cite{b12Bayram,K-Guo2018,c29Guo}. For the methods for solving (\ref{SMP}), many of them (including the classical DR algorithm (\ref{original-DPA}), the adaptive DR algorithm (\ref{iterative2}), ADMM, and many others) can be cast as the Krasnosel'ski\u{i}-Mann fixed-point iteration
\begin{equation}\label{KM-ab1}
    u_{k+1}=\Theta_{k}\mathcal{T}u_{k}+\left(1-\Theta_{k}\right)u_{k},
\end{equation}
where $\Theta_{k}\in [0,1]$ and $\mathcal{T}:\mathbb{H}\to \mathbb{H}$.
For the results on  the convergence analysis for the  Krasnosel'ski\u{i}-Mann algorithm and its variations, we refer the reader to \cite{lipschitz,Cegielski2015,Cegielski2018,Krasn,Kolobov,Mann}. Among the existing  works, a great deal of attention has been paid to the convergence of the  Krasnosel'ski\u{i}-Mann algorithm and its variations. Recently, some papers are devoted to the study of  the convergence rate of the Krasnosel'ski\u{i}-Mann algorithm and its variations.
It was shown in \cite{Baillon,Cominetti} that the Krasnosel'ski\u{i}-Mann algorithm (\ref{KM-ab1}) built from a nonexpansive operator   enjoys the $O(\frac{1}{\sqrt{k}})$ rate of asymptotic regularity. Liang et al. \cite{Liang-2016} established both the pointwise and ergodic
convergence rates, as well as local linear convergence under a metric sub-regularity condition, for the the following inexact Krasnosel'ski\u{i}-Mann fixed-point iteration built from the nonexpansive operator $\mathcal{T}$:
\begin{equation}\label{3discret-ab1}
    u_{k+1}=\Theta_{k}\mathcal{T}u_{k}+\left(1-\Theta_{k}\right)u_{k}+\varepsilon_{k},
\end{equation}
where $\varepsilon_{k}$ is the error of approximating ${T}u_{k}$. As pointed out in \cite{Liang-2016}, the extension of the result of  \cite{Cominetti} to the inexact iteration is quite intricate because  the method of proof in \cite{Cominetti} relies on the recursive bound and exploits some properties of some special functions and an
identity for Catalan numbers, while these recursions are unfortunately not stable to errors. Bravo et al. \cite{Bravo-2019} extended  the result of  \cite{Liang-2016} to general Banach spaces.  For more results on the convergence rate of Krasnosel'ski\u{i}-Mann type algorithms, we refer the reader to \cite{Cegielski2015,Cegielski2018,Kolobov}.

In this paper, linked with the algorithm (\ref{3discret-ab1}), we propose the following nonautonomous evolution equation for solving the fixed point problem associated with $\mathcal{T}$:
\begin{equation}\label{f-dy-system}
 \left\{
             \begin{array}{lr}
             \frac{du}{dt}=\theta(t)[\mathcal{T}(u(t))-u(t)]+f(t)&  \\
             u(t_{0})=u_{0}\in\mathbb{H},&
             \end{array}
   \right.
\end{equation}
where $t_{0}>0$, $\theta:\mathbb{R}_{+}\rightarrow\mathbb{R}_{+}$ is a locally integrable function, and $f:\mathbb{R}_{+}\rightarrow\mathbb{H}$ is a locally integrable operator as a perturbation or computational error. The nonautonomous evolution equation (\ref{f-dy-system}) with $\mathcal{T}$ being a nonexpansive operator and $\theta(t)\equiv 1$ is just the nonautonomous evolution equation considered in \cite[Section 5]{Bravo-2019}. When $\mathcal{T}$ is nonexpansive and $f(t)\equiv 0$, (\ref{f-dy-system}) reduces to the dynamical system studied in \cite{c29Bot}. Historically, the study of continuous time dynamical systems for solving optimization problems
can be traced back at least to 1950s \cite{Arrow1957}. For more results on dynamical system approaches for solving optimization problems and related problems, we refer to \cite{c29Bot,c23Abbas,c24Abbas,c22Attouch,c26Bolte,c30Bot,c31Bot,2019Csetnek,ZM-OP}.

This paper is devoted to  investigating the asymptotic behavior of the trajectory of  the nonautonomous evolution equation (\ref{f-dy-system}) with $\mathcal{T}$ being a quasi-nonexpansive operator as the time tends to infinity. The results obtained are applied to analyze the convergence of the trajectories of the following  adaptive Douglas-Rachford dynamical system:
\begin{equation}\label{dynamic-system}
\left\{
             \begin{array}{lr}
            \frac{du}{dt}+\theta(t)\left[u(t)-R_{\delta B}^{\mu}R_{\gamma A}^{\lambda}(u(t))\right]=f(t), &  \\
             u(t_{0})=u_{0}\in\mathbb{H}.&
             \end{array}
   \right.
   \end{equation}
Let us mention that the above adaptive Douglas-Rachford dynamical system is linked with a relaxed and inexact version of the adaptive DR algorithm (\ref{iterative2}). Indeed, the explicit discretization of the  system (\ref{dynamic-system}) with respect to the time variable $t$, with step size $\Delta t_{k}>0$, yields the iterative scheme
\begin{equation}\label{discret-ds1}
   \frac{ u_{k+1}-u_{k}}{\Delta t_{k}}=\theta_{k}\left[R_{\delta B}^{\mu}R_{\gamma A}^{\lambda}(u_{k})-u_{k}\right]+f_{k}.
\end{equation}
After transposition and setting $\epsilon_{k}= \theta_{k}\Delta t_{k}$, the scheme (\ref{discret-ds1}) becomes
\begin{equation}
    \label{discretization-system}
   u_{k+1}=(1-\epsilon_{k})u_{k}+\epsilon_{k}R_{\delta B}^{\mu}R_{\gamma A}^{\lambda}(u_{k})+\Delta t_{k}f_{k}.
\end{equation}

Main contributions of this paper can be summarized as follows.
\begin{enumerate}
  \item [a)] We show that the trajectory of the nonautonomous evolution equation (\ref{f-dy-system}) governed by a quasi-nonexpansive operator satisfying demiclosedness principle converges weakly to a fixed point of the operator.  A flexible global exponential-type convergence rate is achieved under a metric subregularity condition. Related with the result of \cite{Liang-2016}  on the convergence rate of the inexact Krasnosel'ski\u{i}-Mann fixed-point iteration built from the nonexpansive operator, we derive continuous time analogs  for the asymptotic behaviors of the nonautonomous evolution equation governed by a quasi-nonexpansive operator.

  \item [b)] We study the adaptive Douglas-Rachford dynamical system (\ref{dynamic-system}) which endows the adaptive DR algorithm (\ref{iterative2}) proposed by Dao and Phan \cite{Dao-Phan2018} with computational errors and the continuous time behavior.  With suitable choices of parameters, the adaptive Douglas-Rachford dynamical system  is applied to solve the  problem (\ref{SMP}) in which one operator is strongly monotone and the other one is weakly monotone. The results obtained can be viewed as continuous time analogs with errors to the corresponding results of Dao and Phan \cite{Dao-Phan2018}.
  \end{enumerate}

The rest of this paper is organized as follows.
In Section \ref{sec.2}, we recall some notions and preliminary results for further analysis.
In Section \ref{sec.3}, we analyze the global weak convergence of the nonautonomous evolution equation (\ref{f-dy-system}), and establish a flexible globlal exponential-type convergence rate of the trajectory under a metric subregularity condition. In Section \ref{sec.4}, the results obtained in Section  \ref{sec.3} are applied to analyze the asymptotic behavior of  the adaptive Douglas-Rachford dynamical system (\ref{dynamic-system}). Finally,  the adaptive Douglas-Rachford dynamical system (\ref{dynamic-system}) is applied to solve a ``Strongly+Weakly" convex minimization problem.

\section{Preliminaries}
\label{sec.2}
In what follows, we always use  $w-\lim\limits_{t\rightarrow+\infty}x(t)=x^{*}$ to indicate that  $x(t)$
converges weakly to $x^{*}$ as $t\rightarrow+\infty$.

\begin{definition}\label{continuity}(See, e.g., \cite[Definition 4.1]{lipschitz} or \cite[Definition 9.1]{stronglyconvex})
A function $h: \Omega\subset\mathbb{H}\rightarrow\mathbb{H}$ is said to be
\begin{itemize}
\item[(i)]
 Lipschitz
continuous with constant $L>0$ on $\Omega$ if
\begin{equation*}
   \|h(y)-h(x)\|\leq L\|y-x\|,\quad \forall x,y\in \Omega;
\end{equation*}
\item[(ii)] nonexpansive on $\Omega$ if it is Lipschitz continuous with constant $1$ on $\Omega$,
i.e.,
\begin{equation*}
   \|h(y)-h(x)\|\leq \|y-x\|,\quad \forall x,y\in \Omega;
\end{equation*}
\item[(iii)] quasi-nonexpansive on $\Omega$ if
\begin{equation*}
   \|h(y)-x\|\leq \|y-x\|,\quad\quad\forall y\in \Omega,\forall x\in\text{Fix} (h):=\{x\in \Omega|h(x)=x\};
\end{equation*}
\item[(iv)] uniformly continuous on $\Omega$ if
for every real number $\varepsilon>0$ there exists $\varrho>0$ such that
for every $x, y \in \Omega$ with $\|x-y\|<\varrho$, we have that
$$\|h(x)-h(y)\|<\varepsilon.$$
\end{itemize}
\end{definition}

\begin{definition}\label{absolu-con}(See, e.g. \cite[Definition 2.1]{c22Attouch})
A function $F : [0, b]\rightarrow \mathbb{H}$ (where $b >0$) is said to
be absolutely continuous if one of the following equivalent properties holds:
\begin{itemize}
  \item [(i)] there exists an integrable function $G :[0, b]\rightarrow \mathbb{H}$ such that
  $$
F(t)=F(0)+\int_{0}^{t} G(s) ds, \quad \forall t \in[0, b];$$
 \item [(ii)]$F$ is continuous and its distributional derivative is Lebesgue integrable on $[0, b]$;

  \item [(ii)] for every $\varepsilon> 0$, there exists $\varrho > 0$
  such that for any finite family of intervals $I_{k} =(a_{k}, b_{k} )\subset [0,b]$ we have the implication:
  $$
\left(I_{k} \cap I_{j}=\emptyset \text{ and } \sum_{k}\left|b_{k}-a_{k}\right|<\varrho\right) \Longrightarrow \sum_{k}\left\|F\left(b_{k}\right)-F\left(a_{k}\right)\right\|<\varepsilon.
$$
\end{itemize}
\end{definition}

\begin{remark}\label{remark-1} Let us recall some basic nature of absolutely continuous functions:
\begin{itemize}
\item [(a)] In the light of the definitions above, Lipschitz
continuity implies absolute continuity, which gives rise to uniform continuity.
An absolutely continuous function is differentiable
almost everywhere.
\item [(b)]If $F_{1} : [0, b]\rightarrow \mathbb{H}$ is absolutely continuous and $F_{2}: \mathbb{H}\rightarrow \mathbb{H}$ is Lipschitz
continuous with constant $L$, then their composition function $F = F_{2}\circ F_{1}$ is absolutely continuous.
Moreover, $F$
is almost everywhere differentiable and the inequality
$\left\|F'(\cdot)\right\| \leq L\left\|F_{1}'(\cdot)\right\|$ holds almost
everywhere. (See, e.g. \cite[Remark 1]{c29Bot}).
\end{itemize}

\end{remark}

\begin{definition}\label{a-monotonicity}(See, e.g., \cite[Definition 3.1]{Dao-Phan2018}) A set-valued operator $A: \mathbb{H}\rightrightarrows\mathbb{H}$
is said to be $\alpha$-monotone with  constant $\alpha\in \mathbb{R}$ if,
\begin{equation*}\label{a-monotone}
\langle x-y, u-v\rangle \geq \alpha\|x-y\|^{2},\qquad \forall(x, u),(y, v) \in \text{gra}~A,
\end{equation*}
where
$$\text{gra}~A:=\{(x,u):u\in Ax\}.$$
Moreover, we also say that A is maximally $\alpha$-monotone
if it is $\alpha$-monotone and there is no $\alpha$-monotone operator whose graph strictly contains gra $A$.
\end{definition}

Apparently, $A$ is (resp. maximally) $\alpha$-monotone if and only if $A-\alpha I$ is (resp. maximally)
monotone. We also note that if $\alpha>0$, $\alpha=0$, $\alpha<0$, then
$\alpha$-monotonicity can be referred  as strong monotonicity, monotonicity and weak monotonicity, respectively.
In \cite[Example 12.28]{stronglyconvex}, the weak monotonicity is also called hypomonotonicity.
It was shown in  \cite{Dao-Phan2018} that $\alpha$-monotonicity of a single-valued operator along with continuity leads to maximal $\alpha$-monotonicity.
For detailed discussions on maximal
monotonicity and its variants as well as the connection to optimization problems, we refer the reader
to \cite{lipschitz,Borwein2010}.

The following lemma comes from \cite[Proposition 3.4, Corollary 3.11, Corollary 3.12]{Dao-Phan2018}.
\begin{lemma}\label{prox-lip1}
Let $A : \mathbb{H}\rightrightarrows\mathbb{H}$ be a
maximally $\alpha$-monotone operator and $R :=(1-\nu) \mathrm{I}+\nu J_{\gamma A}$ with $\nu,\gamma>0$.
Suppose that $1+\gamma \alpha>0$. Then the following statements are true:
\begin{itemize}
  \item [(i)] $J_{\gamma A}$ is single-valued and Lipschitz continuous (with constant $\frac{1}{1+\gamma \alpha}$).

  \item [(ii)]If $\nu\geq1$, then, for all $x, y \in\text{ dom } J_{\gamma A}:=\{z\in \mathbb{H}\mid J_{\gamma A}(z)\neq\emptyset \}$,
  $$\|R x-R y\|^{2} \leq (1-\nu)^{2}\|x-y\|^{2}+\nu[(1-\nu)(2+2\gamma \alpha)+\nu]\|J_{\gamma A}x-J_{\gamma A}y\|^{2}.$$

  \item [(iii)] If $A$ is single-valued and Lipschitz continuous with constant $l$, and $$\nu(1+2 \gamma \alpha)-2(1+\gamma \alpha) \geq 0,$$
 then $R$ is Lipschitz continuous with constant
$$
\sqrt{(\nu-1)^{2}-\frac{\nu((\nu-1)(2+2 \gamma \alpha)-\nu)}{1+2 \gamma \alpha+\gamma^{2} l^{2}}}.
$$
\end{itemize}
\end{lemma}

\begin{lemma}\label{V-bounded}{\rm(\cite[Lemme A.5]{1972Brezis})}
Let  $F :[a, b]\rightarrow [0, +\infty)$ be an integrable function and $G :[a, b]\rightarrow \mathbb{R}$ be a continuous function. Suppose that $c \geq 0$
 and
$$
\frac{1}{2} G^{2}(t) \leq \frac{1}{2} c^{2}+\int_{a}^{b} F(s) G(s) ds
$$
for all $t\in[a, b]$. Then $|G(t)| \leq c+\int_{a}^{b} F(s) ds$ for all $t\in[a, b]$.
\end{lemma}

\begin{lemma}\label{V-existence}{\rm(\cite[Lemma 5.1]{c23Abbas})}
Let  $F :[t_{0}, +\infty)\rightarrow \mathbb{R}$ be a locally absolutely continuous, bounded
below function, and  $G\in\mathbb{L}^{1}([t_{0},+\infty))$. Suppose that
$$
\frac{d}{d t} F(t) \leq G(t)
$$
for almost all $t$.
Then there exists $\lim\limits_{t \rightarrow +\infty} F(t) \in \mathbb{R}$.
\end{lemma}


\begin{lemma}\label{Haraux} {\rm(See, e.g., \cite[Lemma 1.2.2]{Haraux2015})}
Let $t_{0}\in\mathbb{R}$ and $h :[t_{0}, +\infty)\rightarrow \mathbb{H}$ be a uniformly continuous function.
If $\int^{+\infty}_{t_{0}}\|h(s)\|ds<+\infty$,
then $$\lim\limits_{t\rightarrow+\infty} h(t)=0.$$
\end{lemma}

\begin{lemma}\label{opial} {\rm(See, e.g., \cite[Lemma 5.3]{c23Abbas})}
Let $\Omega\subseteq\mathbb{H}$ be a nonempty set and $x :[t_{0}, +\infty)\rightarrow \mathbb{H}$ be a given map.
Suppose
\begin{itemize}
    \item [(i)] for every $z\in \Omega$, $\lim\limits_{t\rightarrow+\infty}\|x(t)-z\|$ exists;
    \item [(ii)] every weak sequential cluster point of the map $x$ belongs to $\Omega$.
  \end{itemize}
Then there exists $x_{\infty}\in \Omega$ such that $w-\lim\limits_{t\rightarrow+\infty}x(t)=x_{\infty}$.
\end{lemma}

\begin{definition}\label{metrically-subregular}(See \cite[p.183]{Metric-subregularity})
 A set-valued operator $F:\mathbb{H}\rightrightarrows\mathbb{H}$ is said to be metrically subregular at $z^{*}$ for $y^{*}$ with modulus $\kappa$ if $y^{*}\in F(z^{*})$ and
there exists $\kappa\geq0$,
along with an open ball $B(z^{*},r):=\{y\in\mathbb{H}\mid\|y-z^{*}\|<r\}$, such that
\begin{equation*}\label{metrical-sub}
\text{dist}(z,F^{-1}(y^{*}))\leq \kappa\cdot\text{dist}(y^{*},F(z)),\quad\forall\,z\in B(z^{*},r),
\end{equation*}
where $\text{dist}(z,\Omega)=\min\limits_{y\in\Omega}\|z-y\|$.
\end{definition}

\begin{remark}
For a set-valued operator $F$ and a vector $y^{*}$, as Dontchev and Rockafellar pointed out in \cite{Metric-subregularity}, metric subregularity gives an estimate
for how far a point $z$ is from being a solution to the generalized equation $F(z)\ni y^{*}$
in terms of the ``residual" $\text{dist}(y^{*},F(z))$.
The constant $\kappa$ measures the stability under perturbations of inclusion $y^{*}\in F(z)$.
Metric subregularity and its variants have been widely to establish linear convergence rates of the algorithms.  See e.g. \cite{c29Guo,Cegielski2015,Cegielski2018,Liang-2016,Bravo-2019}.
\end{remark}

\begin{definition}\label{demi-closed}(See, e.g., \cite[Definition 4.26]{lipschitz})
We say that an operator $F: \mathbb{H}\rightarrow\mathbb{H}$  satisfies the demiclosedness principle if $I-F$ is  demiclosed at $0$, i.e., for any sequence $\{x_k\}$,
 $$ w-\lim\limits_{k\rightarrow+\infty}x_k=x^* \text{ and } \|F(x_{k})-x_k\| \to 0 \Rightarrow F(x^{*})=x^*.$$

\end{definition}

\begin{remark}\label{demiclose}
Every nonexpansive operator satisfies the demiclosedness principle(see, e.g., \cite[Theorem 4.27]{lipschitz}), while a quasi-nonexpansive operator need not to satisfy such property. In fact, the demiclosedness principle plays an important role in the convergence analysis for fixed point iteration algorithms built from quasi-nonexpansive operators (see e.g., \cite{Cegielski2015,Cegielski2018,Kolobov}). A closely related concept is the weak regularity introduced in \cite[Definition 3.1]{Cegielski2018} and \cite[Definition 12]{Kolobov} which implies the demiclosedness principle. Further examples of quasi-nonexpansive operators satisfying the demiclosedness principle can be found in \cite[Lemma 3.1]{Wang2011} and \cite[Lemma 4.1]{CegielskiJOTA}.
\end{remark}

Before we finish this section, let us recall a useful identity,
which will be used several times in the following sections.
Its proof is straightforward and so we omit it. 
For all $x,y\in\mathbb{H}$ and all $\epsilon,\varrho\in\mathbb{R}$,
\begin{equation}\label{a-identity}
\|\epsilon x+\varrho y\|^{2}=\epsilon(\epsilon+\varrho)\|x\|^{2}+\varrho(\epsilon+\varrho)\|y\|^{2}-\epsilon \varrho\|x-y\|^{2}.
\end{equation}

Now we give a global weak convergence result on the trajectory.

\section{Asymptotic analysis for nonautonomous evolution  equations}
\label{sec.3}
In this section, we analyze the convergence of the trajectory of the nonautonomous evolution  equation (\ref{f-dy-system}). In what follows, unless otherwise stated,  we always assume that $\theta:[t_{0},+\infty)\rightarrow\mathbb{R}_{+}$  is locally integrable and $f\in\mathbb{L}^{1}([t_{0},+\infty))$. Moreover, the following two assumptions about $\theta$ are used
in the rest of the paper and refer to them when appropriate.
\begin{equation*}
(\textbf{A1}): \inf\limits_{t\in[t_{0},+\infty)}\theta(t)>0;\qquad(\textbf{A2}):\int^{+\infty}_{t_{0}}\theta(s)ds=+\infty.
\end{equation*}
It is clear that (\textbf{A1}) implies (\textbf{A2}).

Recall that  $u: [t_{0},+\infty)\rightarrow\mathbb{H}$ is a strong global solution of the equation (\ref{f-dy-system})
if and only if the following properties are satisfied:
\begin{itemize}
  \item [(i)] $u: \mathbb{R}_{+}\rightarrow\mathbb{H}$ is absolutely continuous on each interval $[t_{0},b]$, $t_{0}<b<+\infty$;

  \item [(ii)]$\frac{du}{dt}=\theta(t)[\mathcal{T}(u(t))-u(t)]+f(t)$ for almost every $t\geq t_{0}$;

  \item [(iii)] $u(t_{0})=u_{0}$.
\end{itemize}

\begin{theorem}\label{3th3ad}
Let $\mathcal{T}: \mathbb{H}\rightarrow\mathbb{H}$ be a quasi-nonexpansive operator with $\text{Fix}(\mathcal{T})\neq\emptyset$ and let $u(t)$ be a strong global solution  of the equation (\ref{f-dy-system}).
Then the following statements are true:
\begin{itemize}
      \item [(i)]$\int^{+\infty}_{t_{0}}\theta(s)\|\mathcal{T}u(s)-u(s)\|^{2}ds<+\infty$.

      \item [(ii)]If $\mathcal{T}$ is uniformly continuous on $\mathbb{H}$ and assumption (\textbf{A1}) holds,
      then
      \begin{equation}\label{ab-th3.1-s0}
      \lim\limits_{t\rightarrow+\infty}\|\mathcal{T}u(t)-u(t)\|=0.
      \end{equation}

      \item [(iii)] If the equation (\ref{ab-th3.1-s0}) holds and  $\mathcal{T}$ satisfies the demiclosedness principle, then there exists $\hat{u}\in \text{Fix}(\mathcal{T})$ such that $w-\lim\limits_{t\rightarrow+\infty}u(t)=\hat{u}$.
    \end{itemize}
\end{theorem}

\begin{proof}
Taking arbitrarily  $u^{*}\in\text{Fix}(\mathcal{T})$, we consider the following auxiliary function:
\begin{equation*}\label{ab-th3.1-s1}
V(t)=\|u(t)-u^{*}\|^{2}.
\end{equation*}
Computing the time derivative of $V(t)$,
we have
\begin{eqnarray}\label{ab-th3.1-s2}
\frac{dV}{dt}&=&2\left\langle u(t)-u^{*}, \frac{du}{dt}\right\rangle\nonumber \\
&=&2\theta(t)\left\langle u(t)-u^{*},\mathcal{T}(u(t))-u(t)\right\rangle+2\langle u(t)-u^{*},f(t)\rangle\nonumber \\
&=&\theta(t)\left(\|\mathcal{T}(u(t))-u^{*}\|^{2}-\|\mathcal{T}(u(t))-u(t)\|^{2}-\|u(t)-u^{*}\|^{2}\right)\nonumber \\
& &+2\langle u(t)-u^{*},f(t)\rangle\nonumber \\
&\leq&-\theta(t)\|\mathcal{T}(u(t))-u(t)\|^{2}+2\langle u(t)-u^{*},f(t)\rangle\nonumber \\
&\leq&-\theta(t)\|\mathcal{T}(u(t))-u(t)\|^{2}+2\|u(t)-u^{*}\|\|f(t)\|,
\end{eqnarray}
where the first inequality is obtained by the quasi-nonexpansiveness of $\mathcal{T}$.
Integrating this inequality from $t_{0}$ to $t$, we have
\begin{eqnarray}\label{ab-th3.1-s3}
V(t)-V(t_{0})
&\leq&-\int^{t}_{t_{0}}\theta(s)\|\mathcal{T}(u(s))-u(s)\|^{2}ds\nonumber\\
&&+2\int^{t}_{t_{0}}\|u(s)-u^{*}\|\|f(s)\|ds,
\end{eqnarray}
which gives
\begin{eqnarray}\label{ab-th3.1-s4}
V(t)\leq V(t_{0})+ 2\int^{t}_{t_{0}}\|u(s)-u^{*}\|\|f(s)\|ds.
\end{eqnarray}
According to Lemma \ref{V-bounded},
we obtain from (\ref{ab-th3.1-s4}) that
\begin{eqnarray*}\label{ab-th3.1-s5}
\|u(t)-u^{*}\| \leq \|u_{0}-u^{*}\|+\int^{t}_{t_{0}}\|f(s)\|ds,
\end{eqnarray*}
which implies that the trajectory $u(t)$ is bounded on $[t_{0},+\infty)$ due to $f\in\mathbb{L}^{1}([t_{0},+\infty))$.
Hence, there exists $\mathcal{M}\in\mathbb{R}_{++}$ such that for all $t$,
\begin{eqnarray}\label{ab-th3.1-s6}
\|u(t)-u^{*}\|\leq \mathcal{M}.
\end{eqnarray}
In the rest of the proof, we verify the statements (i)-(iii) in order.
\begin{itemize}
  \item [a)]It is from (\ref{ab-th3.1-s3}) and (\ref{ab-th3.1-s6}) that
  \begin{eqnarray*}
  \int^{t}_{t_{0}}\theta(s)\|\mathcal{T}(u(s))-u(s)\|^{2}ds\leq V(u_{0})+2\mathcal{M}\int^{t}_{t_{0}}\|f(s)\|ds.
  \end{eqnarray*}
  Taking $t\rightarrow+\infty$ and noting that $f\in\mathbb{L}^{1}([t_{0},+\infty))$, we get (i).

  \item [b)]Suppose that $\mathcal{T}$ is uniformly continuous on $\mathbb{H}$ and the assumption (\textbf{A1}) holds.
  Then, the latter along with (i) yields
  \begin{eqnarray*}
  \int^{+\infty}_{t_{0}}\|\mathcal{T}u(s)-u(s)\|^{2}ds<+\infty.
  \end{eqnarray*}
  On the other hand,
  by the absolute continuity of the trajectory $u(t)$, $\mathcal{T}u(t)$ is uniformly continuous with respect to $t$, so is $\mathcal{T}u(t)-u(t)$.
  Consequently, according to Lemma \ref{Haraux}, we deduct (\ref{ab-th3.1-s0}) directly,
  and so (ii) holds.
  \item [c)] We verify the last assertion via Lemma \ref{opial}.
  Firstly, suppose that the equation (\ref{ab-th3.1-s0}) holds, and that $\mathcal{T}$ satisfies the demiclosedness principle.
  From (\ref{ab-th3.1-s2}) and (\ref{ab-th3.1-s6}),
  we have
   $$\frac{d}{dt}\|u(t)-u^{*}\|^{2}\leq2\mathcal{M}\|f(t)\|,$$
  which gives rise to the existence of
  $\lim\limits_{n\rightarrow+\infty}\|u(t)-u^{*}\|$ by Lemma \ref{V-existence}.
  Since $u^{*}\in \text{Fix} (\mathcal{T})$ has been chosen arbitrary,
  the first assumption in Lemma \ref{opial} is fulfilled.
  On the other hand,
  in view of the boundedness of the trajectory $u(t)$, let $\hat{u}\in\mathbb{H}$ be a weak sequential
  cluster point of $u(t)$, that is, there exists a sequence $t_{n}\rightarrow+\infty$ (as $n\rightarrow+\infty$) such that
  $w-\lim\limits_{n\rightarrow+\infty}u(t_{n})=\hat{u}$.
  Applying the demiclosedness principle of $\mathcal{T}$ and (\ref{ab-th3.1-s0}),
  we have $\hat{u}\in\text{Fix}\mathcal{T}$, and the assertion (iii) follows from Lemma \ref{opial}.
  The proof is complete.
\end{itemize}
\end{proof}

\begin{remark}
Theorem \ref{3th3ad} generalizes \cite[Theorem 6]{c29Bot} where the weak convergence of the trajectory of (\ref{f-dy-system}) with $\mathcal{T}$ being nonexpansive and $f(t)\equiv 0$ was established. The extension of \cite[Theorem 6]{c29Bot} to  Theorem \ref{3th3ad} is nontrivial because  the method of proof   of \cite[Theorem 6]{c29Bot} relies on the decreasing of $V(t)$ on $[t_0,+\infty)$ to guarantee the existence of $\lim_{t\to\infty} V(t)$ while the proof of Theorem \ref{3th3ad} relies on  Lemma \ref{V-existence} since  the decreasing of $V(t)$ is unfortunately not satisfied when $f(t)\not\equiv 0$.
\end{remark}
%
%
%

The following result is concerned with the weak convergence of the equation (\ref{f-dy-system}) when $\mathcal{T}$ has a splitting structure.
\begin{theorem}\label{3th3ad3}
Let $u(t)$ be a strong global solution  of the equation (\ref{f-dy-system}) and let $\mathcal{T}_{1}$, $\mathcal{T}_{2}:\mathbb{H}\rightarrow\mathbb{H}$ be two operators. Suppose that  the assumption (\textbf{A1}) holds and the following conditions are satisfied:
\begin{itemize}
  \item [$(H_{1})$] $\mathcal{T}_{1}$, $\mathcal{T}_{2}$ and $\mathcal{T}$ are uniformly continuous, single-valued and have full domain.
  \item [$(H_{2})$] $I-\mathcal{T}=\mu\left(\mathcal{T}_{1}-\mathcal{T}_{2}\right)$, $\mu\neq0$.
  \item [$(H_{3})$] $\text{Fix}(\mathcal{T})\neq\emptyset$.
  \item [$(H_{4})$] For all $x\in \mathbb{H}$ and $x^{*}\in\text{Fix}(\mathcal{T})$,
  \begin{eqnarray}\label{3th3ad-s1}
   \|\mathcal{T}x-x^{*}\|^{2}&\leq &\| x-x^{*}\|^{2}-\omega_{0}\|(I-\mathcal{T}) x \|^{2}\nonumber \\
      &&-\omega_{1}\left\|\mathcal{T}_{1} x-\mathcal{T}_{1} x^{*}\right\|^{2}
       -\omega_{2}\left\|\mathcal{T}_{2}x-\mathcal{T}_{2}x^{*}\right\|^{2},
  \end{eqnarray}
  where $w_{i}\in\mathbb{R},\,i=0,1,2,$ subject to
\begin{equation}\label{3th3ad3-a4}
\begin{array}{ll}
{\mbox { either }}  &{\left\{\omega_{1}=\omega_{2}=0\right\}\quad\mbox{and}\quad \omega_{0}>0};\\
 {\mbox { or }}  &{\left\{\omega_{1}+\omega_{2}>0\, \mbox { and }\,
\omega_{0}+\frac{\omega_{1} \omega_{2}}{\mu^{2}\left(\omega_{2}+\omega_{2}\right)}>0\right\}}.
\end{array}
\end{equation}
\end{itemize}
Then for any $u^{*}\in\text{Fix}(\mathcal{T})$,
the following statements hold:
\begin{itemize}
\item[(i)] $\int^{+\infty}_{t_{0}}\|\mathcal{T}u(s)-u(s)\|^{2}ds<+\infty$.

\item[(ii)]$\lim\limits_{t\rightarrow+\infty}\|\mathcal{T}u(t)-u(t)\|=0$.

\item[(iii)] If $\mathcal{T}$ satisfies the demiclosedness principle, then
there exists $\hat{u}\in \text{Fix}(\mathcal{T})$
such that $w-\lim\limits_{t\rightarrow+\infty}u(t)=\hat{u}$.

\item[(iv)]If $\omega_{1}+\omega_{2}>0$, then\\
$\int^{+\infty}_{t_{0}}\left\|\omega_{1}\left(\mathcal{T}_{1}u(s)-\mathcal{T}_{1} u^{*}\right)
+\omega_{2}\left(\mathcal{T}_{2} u(s)-\mathcal{T}_{2}u^{*} \right)\right\|^{2}ds<+\infty$,\\
$\lim\limits_{t\rightarrow+\infty}\left\|\omega_{1}\left(\mathcal{T}_{1}u(t)-\mathcal{T}_{1} u^{*}\right)
+\omega_{2}\left(\mathcal{T}_{2} u(t)-\mathcal{T}_{2}u^{*} \right)\right\|=0$,\\
and $\lim\limits_{t\rightarrow+\infty}\mathcal{T}_{1}u(t)=\mathcal{T}_{1}
u^{*}=\lim\limits_{t\rightarrow+\infty}\mathcal{T}_{2}u(t)=\mathcal{T}_{2}u^{*}$.
\end{itemize}
\end{theorem}

\begin{proof}
Set
$$
\omega_{1}^{\prime} :=\left\{\begin{array}{ll}{\frac{\omega_{1} \omega_{2}}{\omega_{1}+\omega_{2}}} & {\text { if } \omega_{1}+\omega_{2}>0,} \\
 {0} & {\text { if } \omega_{1}=\omega_{2}=0}\end{array} \text { and } \omega_{2}^{\prime} :=\left\{\begin{array}{ll}{\frac{1}{\omega_{1}+\omega_{2}}} & {\text { if } \omega_{1}+\omega_{2}>0} \\
 {0} & {\text { if } \omega_{1}=\omega_{2}=0.}
 \end{array}\right.\right.
$$
Then, by (\ref{3th3ad3-a4}),
\begin{eqnarray}\label{3th3ad-s1-0}
\omega_{0}+\frac{\omega_{1}^{\prime}}{\mu^{2}}>0 \quad\text{ and }\quad \omega_{2}^{\prime} \geq 0.
\end{eqnarray}
For all $x, y\in\mathbb{H}$, we derive from (\ref{a-identity}) and the assumption $(H_{2})$ that
\begin{eqnarray}\label{3th3ad-s1-1}
&&\omega_{1}\left\|\mathcal{T}_{1} x-\mathcal{T}_{1} y\right\|^{2}+\omega_{2}\left\|\mathcal{T}_{2} x-\mathcal{T}_{2} y\right\|^{2} \nonumber\\
&=&\omega_{1}^{\prime}\left\|\left(\mathcal{T}_{1}-\mathcal{T}_{2}\right) x-\left(\mathcal{T}_{1}-\mathcal{T}_{2}\right) y\right\|^{2}\nonumber\\
&&+\omega_{2}^{\prime}\left\|\omega_{1}\left(\mathcal{T}_{1} x-\mathcal{T}_{1} y\right)
+\omega_{2}\left(\mathcal{T}_{2} x-\mathcal{T}_{2} y\right)\right\|^{2} \nonumber\\
&=& \frac{\omega_{1}^{\prime}}{\mu^{2}}\|(I-\mathcal{T}) x-(I-\mathcal{T}) y\|^{2}\nonumber\\
&&+\omega_{2}^{\prime}\left\|\omega_{1}\left(\mathcal{T}_{1} x-\mathcal{T}_{1} y\right)
+\omega_{2}\left(\mathcal{T}_{2} x-\mathcal{T}_{2} y\right)\right\|^{2}.
\end{eqnarray}
This combines with (\ref{3th3ad-s1}) and (\ref{3th3ad-s1-0}) implies that $\mathcal{T}$ is quasi-nonexpansive.
So (i),(ii) and (iii) hold by Theorem \ref{3th3ad}.

Next, to verify (iv), we consider the auxiliary function:
$
V(t)=\|u(t)-u^{*}\|^{2}.
$
Computing the time derivative of $V(t)$,
we have
\begin{eqnarray}\label{3th3ad3-s3}
\frac{dV}{dt}&=&2\left\langle u(t)-u^{*}, \frac{du}{dt}\right\rangle\nonumber \\
&=&2\left\langle u(t)-u^{*}, \theta(t)\left(\mathcal{T}u(t)-u(t)\right)\right\rangle+2\langle u(t)-u^{*},f(t)\rangle\nonumber \\
&=&\left\|u(t)-u^{*}+\theta(t)\left(\mathcal{T}u(t)-u(t)\right)\right\|^{2}-\|u(t)-u^{*}\|^{2}\nonumber \\
& &-\theta^{2}(t)\|\mathcal{T}u(t)-u(t)\|^{2}+2\langle u(t)-u^{*},f(t)\rangle\nonumber \\
&=&\left\|\theta(t)\left(\mathcal{T}u(t)-u^{*}\right)+(1-\theta(t))(u(t)-u^{*})\right\|^{2}-\|u(t)-u^{*}\|^{2}\nonumber \\
& &-\theta^{2}(t)\left\|\mathcal{T}u(t)-u(t)\right\|^{2}+2\langle u(t)-u^{*},f(t)\rangle\nonumber \\
&\overset{(\ref{a-identity})}{=}&\theta(t)\left\|\mathcal{T}u(t)-u^{*}\right\|^{2}+(1-\theta(t))\|u(t)-u^{*}\|^{2}-\|u(t)-u^{*}\|^{2}\nonumber \\
& &-\theta(t)(1-\theta(t))\left\|\mathcal{T}u(t)-u(t)\right\|^{2}-\theta^{2}(t)\left\|\mathcal{T}u(t)-u(t)\right\|^{2}\nonumber \\
& &+2\langle u(t)-u^{*},f(t)\rangle\nonumber \\
&\leq&\theta(t)\left\|\mathcal{T}u(t)-\mathcal{T}(u^{*})\right\|^{2}-\theta(t)\|u(t)-u^{*}\|^{2}\nonumber \\
& &-\theta(t)\left\|\mathcal{T}u(t)-u(t)\right\|^{2}+2\|u(t)-u^{*}\|\|f(t)\|.
\end{eqnarray}
By (\ref{3th3ad-s1}) and (\ref{3th3ad-s1-1}), it ensues that
\begin{eqnarray}\label{3th3ad3-s4}
\frac{dV}{dt}&\overset{(\ref{3th3ad-s1})}{\leq}&-\omega_{1}\theta(t)\left\|\mathcal{T}_{1}u(t)-\mathcal{T}_{1} u^{*}\right\|^{2}
-\omega_{2}\theta(t)\left\|\mathcal{T}_{2}u(t)-\mathcal{T}_{2}u^{*} \right\|^{2}\nonumber \\
& &-\theta(t)\left\|\mathcal{T}u(t)-u(t)\right\|^{2}
+2\|u(t)-u^{*}\|\|f(t)\|\nonumber \\
&\leq&-\theta(t)\left(\omega_{1}\left\|\mathcal{T}_{1}u(t)-\mathcal{T}_{1} u^{*}\right\|^{2}
+\omega_{2}\left\|\mathcal{T}_{2}u(t)-\mathcal{T}_{2}u^{*} \right\|^{2}\right)\nonumber \\
& &-\theta(t)\left\|\mathcal{T}u(t)-u(t)\right\|^{2}
+2\|u(t)-u^{*}\|\|f(t)\|\nonumber \\
&\overset{(\ref{3th3ad-s1-1})}{=}& -\theta(t)\frac{\omega_{1}^{\prime}}{\mu^{2}}\|(I-\mathcal{T}) u(t)
-(I-\mathcal{T}) u^{*}\|^{2}\nonumber\\
&&-\theta(t)\omega_{2}^{\prime}\left\|\omega_{1}\left(\mathcal{T}_{1} u(t)-\mathcal{T}_{1} y\right)+\omega_{2}\left(\mathcal{T}_{2} u(t)-\mathcal{T}_{2} u^{*}\right)\right\|^{2}\nonumber \\
& &-\theta(t)\left\|\mathcal{T}u(t)-u(t)\right\|^{2}+2\|u(t)-u^{*}\|\|f(t)\|\nonumber \\
&=&-\theta(t)\omega_{2}^{\prime}\left\|\omega_{1}\left(\mathcal{T}_{1} u(t)-\mathcal{T}_{1}u^{*}\right)
+\omega_{2}\left(\mathcal{T}_{2}u(t)-\mathcal{T}_{2} u^{*}\right)\right\|^{2}\nonumber\\
&& -\theta(t)\left(1+\frac{\omega_{1}^{\prime}}{\mu^{2}}\right)\|\mathcal{T}u(t)-u(t)\|^{2}+2\|u(t)-u^{*}\|\|f(t)\|.
\end{eqnarray}
If $\omega_{1}+\omega_{2}>0$, then $\omega_{2}^{\prime}=\frac{1}{\omega_{1}+\omega_{2}}>0$.
Similar to the proof of Theorem \ref{3th3ad} (i) and Theorem \ref{3th3ad} (ii),
we get the first two conclusions of the assertion (iv). Let us verify the last one.
Observer that $\mathcal{T}_{1} u^{*}=\mathcal{T}_{2} u^{*}$ thanks to the assumption $(H_{2})$ and $u^{*}\in\text{Fix}(\mathcal{T})$.
It is from (ii) and the assumption $(H_{2})$  that
  $$\lim\limits_{t\rightarrow+\infty}\left[\mathcal{T}_{1}u(t)-\mathcal{T}_{1} u^{*}
-\left(\mathcal{T}_{2} u(t)-\mathcal{T}_{2}u^{*} \right)\right]=\lim\limits_{t\rightarrow+\infty}\frac{1}{\mu}(\mathcal{T}u(t)-u(t))=0.$$
Together with the second conclusion in (iv) and noting that $\omega_{1}+\omega_{2}>0$, we derive $\lim\limits_{t\rightarrow+\infty}\mathcal{T}_{1}u(t)=\mathcal{T}_{1} u^{*}$
and $\lim\limits_{t\rightarrow+\infty}\mathcal{T}_{2} u(t)=\mathcal{T}_{2}u^{*}$. So (iv) follows.
The proof is complete.
\end{proof}

\begin{remark}
In Theorem \ref{3th3ad}  and Theorem \ref{3th3ad3}, we do not consider the existence of  strong global solutions of the equation (\ref{f-dy-system}). In fact, the existence of  strong global solutions of  the equation (\ref{f-dy-system}) can be ensured by  the Cauchy-Lipschitz theorem (see e.g. \cite[Corollary 2.6]{Teschl-stability}) under an additional assumption that  $\mathcal{T}$ is Lipschitz continuous. The existence of solutions of nonautonomous evolution equations governed by quasi-nonexpansive operators without Lipschitz assumptions is a nontrivial topic and it is not our concerning in this paper.
\end{remark}

Next, we investigate the global exponential-type convergence of the trajectory of the equation (\ref{f-dy-system}) under the metric subregularity condition. Given a strong global solution $u(t)$ of  (\ref{f-dy-system}), notice that  under suitable conditions, $\lim\limits_{t\rightarrow+\infty}\|u(t)- u^*\|$ exists for any $u^*\in {\rm Fix}(\mathcal{T})$ and $u(t)$ converges weakly to a fixed point of $\mathcal{T}$(see  Theorem \ref{3th3ad}). Then we have the following result.

\begin{theorem}\label{3th4ad}
Suppose that  $\mathcal{T}: \mathbb{H}\rightarrow\mathbb{H}$ be a quasi-nonexpansive  operator with $\text{Fix}(\mathcal{T})\neq\emptyset$  and that $\mathcal{T}$ satisfies the demiclosedness principle.
Let $u(t)$  be a strong global solution of the equation (\ref{f-dy-system}) and $\hat{u}\in{\rm Fix}(\mathcal{T})$  such that $w-\lim\limits_{t\rightarrow+\infty}u(t)=\hat{u}\in {\rm Fix}(\mathcal{T})$.
If $I-\mathcal{T}$ is metrically subregular at $\hat{u}$ for $0$ with a ball $B(\hat{u},r)$ and modulus $\kappa$, and $r>\lim\limits_{t\rightarrow+\infty}\|u(t)-\hat{u}\|$,
then
there exist $t'\geq t_{0}$ and $M_{1}>0$ such that for all $t\geq t'$
  \begin{eqnarray}\label{3th4.2ad-s1}
  &&{\rm dist}^{2}(u(t),{\rm Fix}(\mathcal{T}))\nonumber\\
  &\leq&
  e^{-\frac{1}{\kappa^{2}}\int^{t}_{t'}\theta(s)ds}
  \left(M_{1}\int^{t}_{t'}\|f(s)\|e^{\frac{1}{\kappa^{2}}\int^{s}_{t'}\theta(x)dx}ds+\|u(t')-\hat{u}\|^{2}\right).
  \end{eqnarray}
\end{theorem}

\begin{proof}
Owing to $r>\lim\limits_{t\rightarrow+\infty}\|u(t)-\hat{u}\|$,
there exists $t'\geq t_{0}$ such that $u(t)\in B(\hat{u},r)$ for all $t\geq t'$.
Since $I-\mathcal{T}$ is metrically subregular at $\hat{u}\in\text{Fix}(\mathcal{T})$ for $0$ with the ball
$B(\hat{u},r)$ and modulus $\kappa$, noting $(I-\mathcal{T})^{-1}(0)=\text{Fix}(\mathcal{T})$,
we obtain
\begin{equation}\label{th4.1-1}
\text{dist}(u(t),\text{Fix}(\mathcal{T}))\leq \kappa\|u(t)-\mathcal{T}(u(t))\|,\quad\forall\,t\geq t'.
\end{equation}
On the other hand,
consider the following function:
\begin{equation*}\label{3th4.3-p2}
V_{P}(t)=\text{dist}^{2}(u(t),\text{Fix}(\mathcal{T}))=\|u(t)-P_{Fix(\mathcal{T})}(u(t))\|^{2}.
\end{equation*}
Since $Fix(\mathcal{T})$ is closed and convex by Proposition 4.23 in \cite{lipschitz}, the metric projection $P_{Fix(\mathcal{T})}$ is well defined.
Note that $P_{Fix(\mathcal{T})}(u(t))\in\text{Fix}(\mathcal{T})$ and $$\frac{dV_{P}}{dt}=2\left\langle u(t)-P_{Fix(\mathcal{T})}(u(t)),\frac{du}{dt}\right\rangle$$
by Corollary 12.31 in \cite{lipschitz}.
Therefore, we can replace $\hat{u}$ by $P_{Fix(\mathcal{T})}(u(t))$ in (\ref{ab-th3.1-s2})
to obtain
\begin{eqnarray*}\label{3th4.3-2-2}
\frac{dV_{P}}{dt}
&\leq&-\theta(t)\|\mathcal{T}u(t)-u(t)\|^{2}
+2\left\|u(t)-P_{Fix(\mathcal{T})}(u(t))\right\|\|f(t)\|\nonumber \\
&\overset{(\ref{th4.1-1})}{\leq}&-\frac{\theta(t)}{\kappa^{2}}V_{P}(t)+2\left\|u(t)-P_{Fix(\mathcal{T})}(u(t))\right\|\|f(t)\|\nonumber \\
&\overset{(\ref{ab-th3.1-s6})}{\leq}&-\frac{\theta(t)}{\kappa^{2}}V_{P}(t)+2\mathcal{M}\|f(t)\|,\quad\forall\,t\geq t'.
\end{eqnarray*}
Multiplying this equation by $e^{\frac{1}{\kappa^{2}}\int^{t}_{t'}\theta(s)ds}$,
  and rearranging the terms, we obtain
  \begin{eqnarray*}
  \frac{dV_{P}}{dt}e^{\frac{1}{\kappa^{2}}\int^{t}_{t'}\theta(s)ds}
  +\frac{\theta(t)}{\kappa^{2}}V_{P}(t)e^{\frac{1}{\kappa^{2}}\int^{t}_{t'}\theta(s)ds}
  \leq 2\mathcal{M}\|f(t)\|e^{\frac{1}{\kappa^{2}}\int^{t}_{t'}\theta(s)ds},
  \end{eqnarray*}
  that is,
  \begin{eqnarray*}
  \frac{d}{dt}\left(V_{P}(t)e^{\frac{1}{\kappa^{2}}\int^{t}_{t'}\theta(s)ds}\right)
  \leq 2\mathcal{M}\|f(t)\|e^{\frac{1}{\kappa^{2}}\int^{t}_{t'}\theta(s)ds}.
  \end{eqnarray*}
  Integrating from $t'$ to $t$, we have
  \begin{eqnarray*}
  V_{P}(t)e^{\frac{1}{\kappa^{2}}\int^{t}_{t'}\theta(s)ds}
  \leq 2\mathcal{M}\int^{t}_{t'}\|f(s)\|e^{\frac{1}{\kappa^{2}}\int^{s}_{t'}\theta(x)dx}ds
  +V_{P}(t'),
  \end{eqnarray*}
  which contributes to (\ref{3th4.2ad-s1}) with
  $M_{1}=2\mathcal{M}$.
This completes the proof.\end{proof}

\begin{remark}\label{3remark-rate} (i) When $\mathbb{H}$ is finite-dimensional, the condition $r>\lim\limits_{t\rightarrow+\infty}\|u(t)-\hat{u}\|$
  is satisfied automatically. (ii) Formula (\ref{3th4.2ad-s1}) gives a flexible exponential convergence rate (by  suitable choices of $\theta$ and $f$):
\begin{itemize}
    \item [a)] If $\theta(t)\equiv \Theta>0$ and $f(t)=\frac{1}{t^{p}}e^{-\frac{\Theta}{\kappa^{2}}t}$, $p>1$, then
  $${\rm dist}(u(t),{\rm Fix}(\mathcal{T}))
  =O\left(e^{-\frac{\Theta}{2\kappa^{2}}t}\right),$$
  which is an exponential convergence rate.
  \item [b)] If $\theta(t)=t^{m}$ with $m>-1$ and $f(t)=\frac{1}{t^{p}}e^{-\frac{t^{m+1}}{\kappa^{2}(m+1)}}$ with $p>1$, then
  $${\rm dist}(u(t),{\rm Fix}(\mathcal{T}))
  =O\left(e^{-\frac{t^{m+1}}{2\kappa^{2}(m+1)}}\right).$$
  \item [c)] If $\theta(t)=\frac{1}{t}$ and $f(t)=t^{-\left(\kappa^{-2}+p\right)}$ with $p>1$,
  then
  $${\rm dist}(u(t),{\rm Fix}(\mathcal{T}))=O\left(t^{-\frac{1}{2\kappa^{2}}}\right).$$
\end{itemize}
\end{remark}

\begin{remark}
Recently, Liang et al. \cite{Liang-2016} presented a convergence rate analysis for the inexact Krasnosel'ski\u{i}-Mann iteration algorithm (\ref{3discret-ab1}) with $\mathcal{T}$ being a nonexpansive operator under some restrictive conditions. Under the metric subregularity condition, they demonstrated that the inexact Krasnosel'ski\u{i}-Mann iteration algorithm enjoys a local line convergence rate (see \cite[Theorem 3]{Liang-2016}). As a comparison,  Theorem \ref{3th4ad} gives  a flexible  global exponential-type convergence rate for the nonautonomous evolution equation governed by a quasi-nonexpansive operator  under mild conditions.
\end{remark}

The following result is an immediate corollary of Theorem \ref{3th4ad}, where we consider $f\equiv0$.
\begin{corollary}\label{3co4ad}
Suppose that  $\mathcal{T}: \mathbb{H}\rightarrow\mathbb{H}$ is a quasi-nonexpansive  operator with $\text{Fix}(\mathcal{T})\neq\emptyset$  and that $\mathcal{T}$ satisfies the demiclosedness principle.
Let $u(t)$  be a strong global solution of the equation (\ref{f-dy-system}) with $f(t)\equiv 0$ and $\hat{u}\in{\rm Fix}(\mathcal{T})$ such that $w-\lim\limits_{t\rightarrow+\infty}u(t)=\hat{u}\in {\rm Fix}(\mathcal{T})$.
If $I-\mathcal{T}$ is metrically subregular at $\hat{u}$ for $0$ with a ball $B(\hat{u},r)$ and modulus $\kappa$, and $r>\lim\limits_{t\rightarrow+\infty}\|u(t)-\hat{u}\|$,
then there exists $t'\geq t_{0}$ such that for all $t\geq t'$
$${\rm dist}(u(t),{\rm Fix}(\mathcal{T}))
  \leq
  \|u(t')-\hat{u}\|e^{-\frac{1}{2\kappa^{2}}\int^{t}_{t'}\theta(s)ds}.\\
  $$
\end{corollary}

\section{Results for adaptive Douglas-Rachford dynamical systems}
\label{sec.4}

In this section we shall investigate asymptotic behavior of the trajectories of the adaptive Douglas-Rachford dynamical system  (\ref{dynamic-system}). The results obtained can be viewed  continuous time analogs with errors to the corresponding results of Dao and Phan \cite{Dao-Phan2018}.

\subsection{Existence and uniqueness of the trajectory }
Recall
\begin{equation}\label{iterative}
\begin{aligned}
J_{\gamma A}=(\operatorname{I}+\gamma A)^{-1}, \quad& R_{\gamma A}^{\lambda}=(1-\lambda) \mathrm{I}+\lambda J_{\gamma A},
\\ J_{\delta B}=(\mathrm{I}+\delta B)^{-1}, \quad& R_{\delta B}^{\mu}=(1-\mu) \mathrm{I}+\mu J_{\delta B}.
\end{aligned}
\end{equation}
Let us set
\begin{equation*}
\mathcal{T}^{\epsilon}:=(1-\epsilon)I+\epsilon R_{\delta B}^{\mu}R_{\gamma A}^{\lambda},\quad\epsilon\in(0,1).
\end{equation*}
Then ${\rm Fix} (\mathcal{T}^{\epsilon})={\rm Fix} (R_{\delta B}^{\mu}R_{\gamma A}^{\lambda})$.
The next proposition indicates that $\mathcal{T}^{\epsilon}$ and $R_{\delta B}^{\mu}R_{\gamma A}^{\lambda}$ have certain nice property  when the parameters are properly tuned.
Let us mention that
the first assertion of the proposition follows from \cite[Proposition 3.4]{Dao-Phan2018}, (ii) and (iii) from \cite[Lemma 4.1]{Dao-Phan2018},
and the last one from
\cite[Proposition 4.3]{Dao-Phan2018}.
\begin{proposition}\label{unique1}
Let $A : \mathbb{H}\rightrightarrows\mathbb{H}$ and $B : \mathbb{H}\rightrightarrows\mathbb{H}$ be
respectively maximally $\alpha$- and $\beta$-monotone.
Suppose that the parameters $\gamma,\delta,\lambda,\mu$ in (\ref{iterative})
satisfy
\begin{equation}\label{cond3.1}
\left\{
\begin{aligned}
&\min \{1+\gamma \alpha, 1+\delta \beta\}>0,\\
&(\lambda-1)(\mu-1)=1,\quad\text{and}\quad\delta=(\lambda-1) \gamma.
\end{aligned}
\right.
\end{equation}
Then the following statements are true:
\begin{itemize}
  \item [(i)] $J_{\gamma A}$, $J_{\delta B}$, and $R_{\delta B}^{\mu}R_{\gamma A}^{\lambda}$ are single-valued and have
full domain.
\item [(ii)]$I-\mathcal{T}^{\epsilon}=\epsilon(\mathrm{I}-R_{\delta B}^{\mu}R_{\gamma A}^{\lambda})=\epsilon\mu\left(J_{\gamma A}-J_{\delta B} R_{\gamma A}^{\lambda}\right)$.
\item [(iii)] $J_{\gamma A}({\rm Fix} (R_{\delta B}^{\mu}R_{\gamma A}^{\lambda}))=J_{\delta B} R_{\gamma A}^{\lambda}({\rm Fix}  (R_{\delta B}^{\mu}R_{\gamma A}^{\lambda}))=\text{zer}(A+B)$.
\item [(iv)] If $\lambda,\mu\geq 1$, then
for all $x,y\in \mathbb{H}$,
\begin{eqnarray}\label{le3.1-s0}
\|\mathcal{T}^{\epsilon} x-\mathcal{T}^{\epsilon} y\|^{2}&\leq &\|x-y\|^{2}-\frac{1-\epsilon}{\epsilon}\|(I-\mathcal{T}^{\epsilon})x-(I-\mathcal{T}^{\epsilon})y\|^{2} \nonumber\\
&&-\epsilon\mu(2+2 \gamma \alpha-\mu)\left\|J_{\gamma A} x-J_{\gamma A} y\right\|^{2} \nonumber\\
&&-\epsilon\mu(\mu-(2-2 \gamma \beta))\left\|J_{\delta B} R_{\gamma A}^{\lambda} x-J_{\delta B} R_{\gamma A}^{\lambda} y\right\|^{2},
\end{eqnarray}
and
\begin{eqnarray}\label{le3.1-s1}
&&\|R_{\delta B}^{\mu}R_{\gamma A}^{\lambda} x-R_{\delta B}^{\mu}R_{\gamma A}^{\lambda} y\|^{2} \nonumber\\
&\leq &\|x-y\|^{2}-\mu(2+2 \gamma \alpha-\mu)\left\|J_{\gamma A} x-J_{\gamma A} y\right\|^{2} \nonumber\\
&&-\mu(\mu-(2-2 \gamma \beta))\left\|J_{\delta B} R_{\gamma A}^{\lambda} x-J_{\delta B} R_{\gamma A}^{\lambda} y\right\|^{2}.
\end{eqnarray}
\end{itemize}
\end{proposition}

According to Lemma \ref{prox-lip1} and Proposition \ref{unique1}, the operator
$R_{\delta B}^{\mu}R_{\gamma A}^{\lambda}$ is Lipschitz continous if the parameters $\gamma,\delta,\lambda,\mu$
are subject to (\ref{cond3.1}), and
such a property makes for a guarantee to bring about existence and uniqueness of a
strong global solution of the system (\ref{dynamic-system}).
Indeed,
the system (\ref{dynamic-system}) can be rewritten as
\begin{equation*}
 \left\{
             \begin{array}{lr}
             \frac{du}{dt}=F(t,u) &  \\
             u(t_{0})=u_{0}\in\mathbb{H},&
             \end{array}
   \right.
\end{equation*}
where $F:[t_{0},+\infty)\times\mathbb{H}\rightarrow\mathbb{H}$ is defined by $F(t,u)=\theta(t)(R_{\delta B}^{\mu}R_{\gamma A}^{\lambda}u-u)+f(t)$.
Applying the Lipschitz continuity of $R_{\delta B}^{\mu}R_{\gamma A}^{\lambda}$, the local integrability of $\theta(\cdot)$ and $f\in\mathbb{L}^{1}([t_{0},+\infty))$,
we can easily verify that the conditions of the Cauchy-Lipschitz theorem (see e.g.
\cite[Corollary 2.6]{Teschl-stability}) are satisfied.
In this way, we get a strong global solution of the system (\ref{dynamic-system}). In addition, the solution is a
classical solution of class $\mathcal{C}^{1}$ if the functions $\theta(t)$ and $f(t)$ are continuous.

In view of the discussion above, an immediate conclusion follows:


\begin{theorem}\label{th3.1}
Let $A : \mathbb{H}\rightrightarrows\mathbb{H}$ and $B : \mathbb{H}\rightrightarrows\mathbb{H}$ be
respectively maximally $\alpha$- and $\beta$-monotone.
Suppose that (\ref{cond3.1}) holds.
Then for each initial point $u_{0}\in\mathbb{H}$,
there exists a unique strong global solution (trajectory) $u(t)$
of the system {\rm(\ref{dynamic-system})} in the global time interval
$[t_{0},+\infty)$.
\end{theorem}

\subsection{Convergence of the trajectories}

We note that $R_{\delta B}^{\mu}R_{\gamma A}^{\lambda}$ is nonexpansive on $\mathbb{H}$,
provided that all parameters occurring in (\ref{le3.1-s1}) cater for
\begin{equation}\label{para-sec-3}
\mu(2+2 \gamma \alpha-\mu)\geq0\quad\text{and}\quad\mu(\mu-(2-2 \gamma \beta))\geq0.
\end{equation}

Now we are going to discuss that how the parameters play a role in the convergence analysis of the system {\rm(\ref{dynamic-system})}.
Consider the following two parametric options:
\begin{itemize}
  \item [(\textbf{C1})]:
  $\mu=\lambda=2$, $\gamma=\delta\in\mathbb{R}_{++}$, $\epsilon\in(0,1)$ and
  \begin{equation}\label{assump2}
  \alpha+\beta>0\quad\text{and}\quad\frac{\gamma\alpha\beta}{\alpha+\beta}>\epsilon-1.
\end{equation}

  \item [(\textbf{C2})]:$\alpha+\beta\in\mathbb{R}_{+}$ and $(\gamma,\delta,\lambda,\mu)\in \mathbb{R}^{2}_{++}\times[1,+\infty)^{2}$
  satisfy
  \begin{equation}\label{3th3.2-s1}
\left\{
\begin{array}{l}{1+2 \gamma \alpha>0}, \\
{\mu \in[2-2 \gamma \beta, 2+2 \gamma \alpha]},\\
{(\lambda-1)(\mu-1)=1, \text{ and } \delta=(\lambda-1) \gamma}.
\end{array}
\right.
\end{equation}

\end{itemize}
It is clear that (\ref{3th3.2-s1}) implies (\ref{para-sec-3}), and so $R_{\delta B}^{\mu}R_{\gamma A}^{\lambda}$ is nonexpansive
in case (\textbf{C2}).
However, case (\textbf{C1}) alone is not sufficient for such a property to be guaranteed.
Indeed, in this case,
the expression (\ref{le3.1-s1}) reduces to
\begin{equation*}
\|R_{\delta B}^{\mu}R_{\gamma A}^{\lambda} x-R_{\delta B}^{\mu}R_{\gamma A}^{\lambda} y\|^{2} \leq \|x-y\|^{2}-4\gamma\Phi(\alpha,\beta,x,y),
\quad\forall x,y\in\mathbb{H},
\end{equation*}
where $$\Phi(\alpha,\beta,x,y)=\alpha\left\|J_{\gamma A} x-J_{\gamma A} y\right\|^{2}+
\beta\left\|J_{\delta B} R_{\gamma A}^{\lambda} x-J_{\delta B} R_{\gamma A}^{\lambda} y\right\|^{2}.$$
Note that $\Phi(\alpha,\beta,x,y)$ is not necessarily nonnegative even if $\alpha+\beta>0$ (a similar discussion can be found in \cite{c29Guo}).
Thus, some existing results depending on the nonexpansiveness of an operator are not applicable in case (\textbf{C1}).
Fortunately, we notice that
$$\mathcal{T}^{\epsilon}\,\,\mbox{is quasi-nonexpansive}$$
in case (\textbf{C1}) by $(\ref{le3.1-s0})$ with a parallel derivation of (\ref{3th3ad-s1-1}), and that the system {\rm(\ref{dynamic-system})} is actually equivalent to the following system
\begin{equation}\label{tdynamic-system}
\left\{
             \begin{array}{lr}
            \frac{du}{dt}+\frac{\theta(t)}{\epsilon}\left[u(t)-\mathcal{T}^{\epsilon}u(t)\right]=f(t), &  \\
             u(t_{0})=u_{0}\in\mathbb{H}.&
             \end{array}
   \right.
   \end{equation}
   according to (ii) of Proposition \ref{unique1}.
This allows us to fall back on the results of Section \ref{sec.3}.

\begin{proposition}\label{3pro-c12}
Suppose that the parameters $\gamma,\delta,\lambda,\mu$ in (\ref{iterative})
satisfy either  (\textbf{C1}) or (\textbf{C2}).
Then (\ref{cond3.1}) holds.
\end{proposition}
\begin{proof}
See the proof of \cite[Theorem 4.5]{Dao-Phan2018}.
\end{proof}

We are now in position to establish
the weak convergence
of the system (\ref{dynamic-system})
in cases (\textbf{C1}) and (\textbf{C2}).
Note that the sum $A+B$ is strongly monotone in case (\textbf{C1}) due to $\alpha+\beta>0$, which gives rise to that
the problem (\ref{SMP}) has a unique solution.
We then learn from Lemma \ref{prox-lip1} (i), Proposition \ref{unique1} (iii) and Proposition \ref{3pro-c12} that
${\rm Fix}(R_{\delta B}^{\mu}R_{\gamma A}^{\lambda})\neq\emptyset$ in case (\textbf{C1}).

\begin{theorem}\label{3th3.4}
Let $A : \mathbb{H}\rightrightarrows\mathbb{H}$ and $B : \mathbb{H}\rightrightarrows\mathbb{H}$ be
respectively maximally $\alpha$- and $\beta$-monotone.
Suppose that the parameters $\alpha,\beta,\gamma,\delta,\lambda,\mu$ satisfy (\textbf{C1}),
and that the assumption (\textbf{A1}) holds.
Let $u(t)$ be the trajectory of the system {\rm (\ref{dynamic-system})}, and let $u^{*}\in{\rm Fix}(R_{\delta B}^{\mu}R_{\gamma A}^{\lambda})$.
Then the following
statements are true:
\begin{itemize}
\item[(i)]
$\int^{+\infty}_{t_{0}}\|R_{\delta B}^{\mu}R_{\gamma A}^{\lambda}u(s)-u(s)\|^{2}ds<+\infty$.

\item[(ii)]$\lim\limits_{t\rightarrow+\infty}\|R_{\delta B}^{\mu}R_{\gamma A}^{\lambda}u(t)-u(t)\|=0$.

\item[(iii)] If $R_{\delta B}^{\mu}R_{\gamma A}^{\lambda}$ satisfies the demiclosedness principle, then
there exists $\hat{u}\in \text{Fix}(R_{\delta B}^{\mu}R_{\gamma A}^{\lambda})$ such that
$w-\lim\limits_{t\rightarrow+\infty}u(t)=\hat{u}$.

\item[(iv)]
$\int^{+\infty}_{t_{0}}\left\|\alpha\left(J_{\gamma A}u(s)-J_{\gamma A} u^{*}\right)
+\beta\left(J_{\delta B} R_{\gamma A}^{\lambda} u(s)-J_{\delta B} R_{\gamma A}^{\lambda}u^{*} \right)\right\|^{2}ds<+\infty$.
\item[(v)]
$\lim\limits_{t\rightarrow+\infty}\left\|\alpha\left(J_{\gamma A}u(t)-J_{\gamma A} u^{*}\right)
+\beta\left(J_{\delta B} R_{\gamma A}^{\lambda} u(t)-J_{\delta B} R_{\gamma A}^{\lambda}u^{*} \right)\right\|=0$.
\item[(vi)]
$\lim\limits_{t\rightarrow+\infty}J_{\gamma A}u(t)=J_{\gamma A} u^{*}=\lim\limits_{t\rightarrow+\infty}J_{\delta B} R_{\gamma A}^{\lambda} u(t)=J_{\delta B} R_{\gamma A}^{\lambda}u^{*}=\text{zer}(A+B)$.


\end{itemize}
\end{theorem}

\begin{proof}
Since the system {\rm(\ref{dynamic-system})} is equivalent to the system {\rm(\ref{tdynamic-system})} and $\mathcal{T}^{\epsilon}$
is quasi-nonexpansive,
this theorem can be verified by Theorem \ref{3th3ad3} with $\mathcal{T}_{1}=J_{\gamma A}$, $\mathcal{T}_{2}=J_{\delta B} R_{\gamma A}^{\lambda}$ and
$\mathcal{T}=\mathcal{T}^{\epsilon}$.
Let us set $w_{1}=4 \epsilon\gamma \alpha$ and $w_{2}=4\epsilon\gamma \beta$.
Then both obey (\ref{3th3ad3-a4}) by (\ref{assump2}).
In order to fulfil the assumptions ($H_{1}$)-($H_{4}$) of Theorem \ref{3th3ad3} it suffices to show that
 (\textbf{C1}) implies (\ref{cond3.1}).
This is immediate by Proposition \ref{3pro-c12}. The proof is complete.
\end{proof}

It turns out from Theorem \ref{3th3.4} (vi) that $J_{\gamma A}u(t)$ and $J_{\delta B} R_{\gamma A}^{\lambda} u(t)$
globally converge to the unique solution of the problem (\ref{SMP}).
In what follows, we turn our attention to case (\textbf{C2}) in which
$R_{\delta B}^{\mu}R_{\gamma A}^{\lambda}$ is nonexpansive.
We learn from Remark \ref{demiclose} that $I-R_{\delta B}^{\mu}R_{\gamma A}^{\lambda}$ is demiclosed at $0$, so is $I-\mathcal{T}^{\epsilon}$.
Let us set $w_{0}=\frac{1-\epsilon}{\epsilon}>0$, $w_{1}=\epsilon\mu(2+2 \gamma \alpha-\mu)\geq0$ and $w_{2}=\epsilon\mu(\mu-(2-2 \gamma \beta))\geq0$.
Then $w_{i}$, $i=0,1,2,$ cater for (\ref{3th3ad3-a4}).
On the other hand, it follows from Proposition \ref{3pro-c12} that  (\textbf{C2}) implies (\ref{cond3.1}).
Thus, the assumptions ($H_{1}$)-($H_{4}$) of Theorem \ref{3th3ad3} with
$\mathcal{T}_{1}=J_{\gamma A}$, $\mathcal{T}_{2}=J_{\delta B} R_{\gamma A}^{\lambda}$ and
$\mathcal{T}=\mathcal{T}^{\epsilon}$ are fulfilled in case (\textbf{C2}), provided that $\text{zer}(A+B)\neq\emptyset$.
By the analysis above, we derive the following theorem immediately.

\begin{theorem}\label{3th3.5-2}
Let $A : \mathbb{H}\rightrightarrows\mathbb{H}$ and $B : \mathbb{H}\rightrightarrows\mathbb{H}$ be
respectively maximally $\alpha$- and $\beta$-monotone.
Suppose that the parameters $\alpha,\beta,\gamma,\delta,\lambda,\mu$ satisfy  (\textbf{C2}),
$\text{zer}(A+B)\neq\emptyset$, and that the assumption (\textbf{A1}) holds.
Let $u(t)$ be the trajectory of the system {\rm (\ref{dynamic-system})}.
Then, for any $u^{*}\in{\rm Fix}(R_{\delta B}^{\mu}R_{\gamma A}^{\lambda})$, the following
statements are true:
\begin{itemize}
\item[(i)]
$\int^{+\infty}_{t_{0}}\|R_{\delta B}^{\mu}R_{\gamma A}^{\lambda}u(s)-u(s)\|^{2}ds<+\infty$.

\item[(ii)]$\lim\limits_{t\rightarrow+\infty}\|R_{\delta B}^{\mu}R_{\gamma A}^{\lambda}u(t)-u(t)\|=0$.

\item[(iii)]
there exists $\hat{u}\in \text{Fix}(R_{\delta B}^{\mu}R_{\gamma A}^{\lambda})$ such that
$w-\lim\limits_{t\rightarrow+\infty}u(t)=\hat{u}$.

\item[(iv)]If $\alpha+\beta>0$, then\\
$\int^{+\infty}_{t_{0}}\left\|w_{1}\left(J_{\gamma A}u(s)-J_{\gamma A} u^{*}\right)
+w_{2}\left(J_{\delta B} R_{\gamma A}^{\lambda} u(s)-J_{\delta B} R_{\gamma A}^{\lambda}u^{*} \right)\right\|^{2}ds<+\infty$,\\
$\lim\limits_{t\rightarrow+\infty}\left\|w_{1}\left(J_{\gamma A}u(t)-J_{\gamma A} u^{*}\right)
+w_{2}\left(J_{\delta B} R_{\gamma A}^{\lambda} u(t)-J_{\delta B} R_{\gamma A}^{\lambda}u^{*} \right)\right\|=0$, and\\
$\lim\limits_{t\rightarrow+\infty}J_{\gamma A}u(t)=J_{\gamma A} u^{*}=\lim\limits_{t\rightarrow+\infty}J_{\delta B} R_{\gamma A}^{\lambda} u(t)=J_{\delta B} R_{\gamma A}^{\lambda}u^{*}=\text{zer}(A+B)$.

\end{itemize}
\end{theorem}

The next theorem serves to show that the system still converges
when the assumption (\textbf{A1}) is weakened  to (\textbf{A2}) and $\epsilon=0$ in (\textbf{C1}),
i.e., $\int^{+\infty}_{t_{0}}\theta(s)ds=+\infty$ and \begin{itemize}
  \item [(\textbf{C3})]:
  $\mu=\lambda=2$, $\gamma=\delta\in\mathbb{R}_{++}$, and
  \begin{equation}\label{3assump}
  \alpha+\beta>0\quad\text{and}\quad\frac{\gamma\alpha\beta}{\alpha+\beta}>-1.
\end{equation}
\end{itemize}
Note that $R_{\delta B}^{\mu}R_{\gamma A}^{\lambda}$ is Lipschitz continuous with a constant $L>0$
 in case (\textbf{C3}).
\begin{theorem}\label{th3.2}
Let $A : \mathbb{H}\rightrightarrows\mathbb{H}$ and $B : \mathbb{H}\rightrightarrows\mathbb{H}$ be
respectively maximally $\alpha$- and $\beta$-monotone.
Suppose that the parameters $\alpha,\beta,\gamma,\delta,\lambda,\mu$ satisfy (\textbf{C3}),
and that the assumption (\textbf{A2}) holds.
Let $u(t)$ be the trajectory of the system {\rm (\ref{dynamic-system})}.
Then the following
statements are true:
\begin{itemize}
\item[(i)]
$\int^{+\infty}_{t_{0}}\theta(s)\|R_{\delta B}^{\mu}R_{\gamma A}^{\lambda}u(s)-u(s)\|^{2}ds<+\infty$.

\item[(ii)]$\lim\limits_{t\rightarrow+\infty}\|R_{\delta B}^{\mu}R_{\gamma A}^{\lambda}u(t)-u(t)\|=0$.

\item[(iii)]If $R_{\delta B}^{\mu}R_{\gamma A}^{\lambda}$ satisfies the demiclosedness principle,
then there exists $\hat{u}\in \text{Fix}(R_{\delta B}^{\mu}R_{\gamma A}^{\lambda})$ such that
$w-\lim\limits_{t\rightarrow+\infty}u(t)=\hat{u}$.
\end{itemize}
\end{theorem}

\begin{proof}
Observe that $R_{\delta B}^{\mu}R_{\gamma A}^{\lambda}$ is single-valued, and has
full domain in that (\textbf{C3}) satisfies (\ref{cond3.1}). Taking $u^{*}\in \text{Fix}(R_{\delta B}^{\mu}R_{\gamma A}^{\lambda})$,
we consider the function
$
V(t)=\|u(t)-u^{*}\|^{2}.
$
Similar to the deducing of (\ref{3th3ad3-s3}) with $\mathcal{T}=R_{\delta B}^{\mu}R_{\gamma A}^{\lambda}$,
we have
\begin{eqnarray*}
\frac{dV}{dt}
&\leq&\theta(t)\left\|R_{\delta B}^{\mu}R_{\gamma A}^{\lambda}u(t)-R_{\delta B}^{\mu}R_{\gamma A}^{\lambda}(u^{*})\right\|^{2}-\theta(t)\|u(t)-u^{*}\|^{2}\nonumber \\
& &-\theta(t)\left\|R_{\delta B}^{\mu}R_{\gamma A}^{\lambda}u(t)-u(t)\right\|^{2}+2\|u(t)-u^{*}\|\|f(t)\|
\nonumber\\
&\overset{(\ref{le3.1-s1})}{\leq}&-4\gamma\alpha\theta(t)\left\|J_{\gamma A}u(t)-J_{\gamma A} u^{*}\right\|^{2}\nonumber\\
&&-4 \gamma \beta\theta(t)\left\|J_{\delta B} R_{\gamma A}^{\lambda}u(t)-J_{\delta B} R_{\gamma A}^{\lambda}u^{*} \right\|^{2}\nonumber \\
& &-\theta(t)\left\|R_{\delta B}^{\mu}R_{\gamma A}^{\lambda}u(t)-u(t)\right\|^{2}
+2\|u(t)-u^{*}\|\|f(t)\|.
\end{eqnarray*}
By (\ref{a-identity}) and Proposition \ref{unique1} (ii), it ensues that
\begin{eqnarray*}
\frac{dV}{dt}
&\leq&-4 \gamma\theta(t)\left( \alpha\left\|J_{\gamma A}u(t)-J_{\gamma A} u^{*}\right\|^{2}
+\beta\left\|J_{\delta B} R_{\gamma A}^{\lambda}u(t)-J_{\delta B} R_{\gamma A}^{\lambda}u^{*} \right\|^{2}\right)\nonumber \\
& &-\theta(t)\left\|R_{\delta B}^{\mu}R_{\gamma A}^{\lambda}u(t)-u(t)\right\|^{2}
+2\|u(t)-u^{*}\|\|f(t)\|\nonumber \\
&\overset{(\ref{a-identity})}{=}& -\frac{\gamma \alpha\beta\theta(t)}{\alpha+\beta}\|(I-R_{\delta B}^{\mu}R_{\gamma A}^{\lambda}) u(t)
-(I-R_{\delta B}^{\mu}R_{\gamma A}^{\lambda}) u^{*}\|^{2}\nonumber\\
&&-\frac{4 \gamma\theta(t)}{ \alpha+\beta}\left\| \alpha\left(J_{\gamma A} u(t)-J_{\gamma A} u^{*}\right)+ \beta\left(J_{\delta B} R_{\gamma A}^{\lambda} u(t)-J_{\delta B} R_{\gamma A}^{\lambda} u^{*}\right)\right\|^{2}\nonumber \\
& &-\theta(t)\left\|R_{\delta B}^{\mu}R_{\gamma A}^{\lambda}u(t)-u(t)\right\|^{2}+2\|u(t)-u^{*}\|\|f(t)\|\nonumber \\
&=&-\frac{4 \gamma\theta(t)}{ \alpha+\beta}\left\|\alpha\left(J_{\gamma A} u(t)-J_{\gamma A} u^{*}\right)+\beta\left(J_{\delta B} R_{\gamma A}^{\lambda} u(t)-J_{\delta B} R_{\gamma A}^{\lambda} u^{*}\right)\right\|^{2}\nonumber \\
&& -\theta(t)\left(1+\frac{\gamma\alpha\beta}{\alpha+\beta}\right)\|R_{\delta B}^{\mu}R_{\gamma A}^{\lambda}u(t)-u(t)\|^{2}+2\|u(t)-u^{*}\|\|f(t)\|\nonumber \\
&\leq& -\theta(t)\left(1+\frac{\gamma\alpha\beta}{\alpha+\beta}\right)\|R_{\delta B}^{\mu}R_{\gamma A}^{\lambda}u(t)-u(t)\|^{2}+2\|u(t)-u^{*}\|\|f(t)\|.
\end{eqnarray*}
Similar to the proof of Theorem \ref{3th3ad} (i), we get the assertion (i) and the boundedness of $u(t)$.
Owing to the Lipschitz continuity of $R_{\delta B}^{\mu}R_{\gamma A}^{\lambda}$,
$R_{\delta B}^{\mu}R_{\gamma A}^{\lambda}u(t)-u(t)$ is bounded
on $[t_{0},+\infty)$.
We learn from Remark \ref{remark-1} (b) that the function $ t\mapsto R_{\delta B}^{\mu}R_{\gamma A}^{\lambda}u(t)$ is almost everywhere differentiable
and $\|\frac{d}{d t} R_{\delta B}^{\mu}R_{\gamma A}^{\lambda}u(t)\|\leq L\|\frac{du(t)}{d t}\|$ holds for almost all $t \geq 0$.
Thus we deduce that
\begin{eqnarray}\label{3th3.2-s3}
&&\frac{d}{d t}\left(\frac{1}{2}\|R_{\delta B}^{\mu}R_{\gamma A}^{\lambda}u(t)-u(t)\|^{2}\right) \nonumber\\
&=&\left\langle\frac{d}{d t} R_{\delta B}^{\mu}R_{\gamma A}^{\lambda}u(t)-\frac{d u(t)}{d t}, R_{\delta B}^{\mu}R_{\gamma A}^{\lambda}u(t)-u(t)\right\rangle \nonumber\\
&=&-\left\langle\frac{d u(t)}{d t}, R_{\delta B}^{\mu}R_{\gamma A}^{\lambda}u(t)-u(t)\right\rangle+\left\langle\frac{d}{d t} R_{\delta B}^{\mu}R_{\gamma A}^{\lambda}u(t), R_{\delta B}^{\mu}R_{\gamma A}^{\lambda}u(t)-u(t)\right\rangle \nonumber\\
&\overset{(\ref{dynamic-system})}{=}&
-\theta(t) \left\| R_{\delta B}^{\mu}R_{\gamma A}^{\lambda}u(t)-u(t)\right\|^{2}-\left\langle f(t),R_{\delta B}^{\mu}R_{\gamma A}^{\lambda}u(t)-u(t)\right\rangle\nonumber\\
&&+\left\langle\frac{d}{d t} R_{\delta B}^{\mu}R_{\gamma A}^{\lambda}u(t), R_{\delta B}^{\mu}R_{\gamma A}^{\lambda}u(t)-u(t)\right\rangle \nonumber\\
&\leq&-\theta(t)\left\|R_{\delta B}^{\mu}R_{\gamma A}^{\lambda}u(t)-u(t)\right\|^{2}
-\left\langle f(t),R_{\delta B}^{\mu}R_{\gamma A}^{\lambda}u(t)-u(t)\right\rangle\nonumber\\
&&+L\left\|\frac{d u(t)}{d t}\right\| \cdot\left\|R_{\delta B}^{\mu}R_{\gamma A}^{\lambda}u(t)-u(t)\right\|\nonumber\\
&\overset{(\ref{dynamic-system})}{\leq}&(L-1)\left(\theta(t)\left\|R_{\delta B}^{\mu}R_{\gamma A}^{\lambda}u(t)-u(t)\right\|+\|f(t)\|\right) \cdot\left\|R_{\delta B}^{\mu}R_{\gamma A}^{\lambda}u(t)-u(t)\right\|\nonumber\\
&&-\left\langle f(t),R_{\delta B}^{\mu}R_{\gamma A}^{\lambda}u(t)-u(t)\right\rangle\nonumber\\
&\leq&(L-1)\theta(t)\left\|R_{\delta B}^{\mu}R_{\gamma A}^{\lambda}u(t)-u(t)\right\|^{2}+2\|f(t)\|\cdot\left\|R_{\delta B}^{\mu}R_{\gamma A}^{\lambda}u(t)-u(t)\right\|,
\end{eqnarray}
that is,
\begin{eqnarray*}
&&\frac{d}{d t}\left(\frac{1}{2}\|R_{\delta B}^{\mu}R_{\gamma A}^{\lambda}u(t)-u(t)\|^{2}\right) \nonumber\\
&\leq&L\theta(t)\left\|R_{\delta B}^{\mu}R_{\gamma A}^{\lambda}u(t)-u(t)\right\|^{2}+2\|f(t)\|\cdot\left\|R_{\delta B}^{\mu}R_{\gamma A}^{\lambda}u(t)-u(t)\right\|,
\end{eqnarray*}
Note that the right side is integrable owing to the boundedness
of $R_{\delta B}^{\mu}R_{\gamma A}^{\lambda}u(t)-u(t)$ and the assertion (i) as well as $f\in\mathbb{L}^{1}([t_{0},+\infty))$. Therefore,
it follows from Lemma \ref{V-existence} that $\lim\limits_{t\rightarrow+\infty}\left\|R_{\delta B}^{\mu}R_{\gamma A}^{\lambda}u(t)-u(t)\right\|$ exists,
and the assertion (ii) holds by the assumption (\textbf{A2}) and the assertion (i).
The last assertion is verified by a similar deducing of Theorem \ref{3th3ad} (iii) with $\mathcal{T}=R_{\delta B}^{\mu}R_{\gamma A}^{\lambda}$.
The proof is complete.
\end{proof}

\begin{remark}\label{remark-th4.3}
For the case:
$$
\lambda=\mu=2,\quad\gamma=\delta \in \mathbb{R}_{++}\quad\text{and}\quad\alpha+\beta>0,
$$
the convergence of the DR algorithm (\ref{iterative2}) researched in \cite{Dao-Phan2018}
requires the following constraint:
$$
\frac{\gamma\alpha\beta}{\alpha+\beta}>\epsilon-1,\quad 0<\epsilon<1.
$$
In contrast, Theorem \ref{th3.2} allows $\epsilon$ to be vanished, that is,
$$\frac{\gamma\alpha\beta}{\alpha+\beta}>-1.$$
\end{remark}

Noticing that $R_{\delta B}^{\mu}R_{\gamma A}^{\lambda}$ is nonexpansive and following a similar proof of Theorem \ref{th3.2}, we have the following conclusion.
\begin{theorem}\label{th3.3}
Let $A : \mathbb{H}\rightrightarrows\mathbb{H}$ and $B : \mathbb{H}\rightrightarrows\mathbb{H}$ be
respectively maximally $\alpha$- and $\beta$-monotone.
Suppose that the parameters $\alpha,\beta,\gamma,\delta,\lambda,\mu$ satisfy (\textbf{C2}), $\text{zer}(A+B)\neq\emptyset$,
and that the assumption (\textbf{A2}) holds.
Let $u(t)$ be the trajectory of the system {\rm (\ref{dynamic-system})}.
Then the following
statements are true:
\begin{itemize}
\item[(i)]
$\int^{+\infty}_{t_{0}}\theta(s)\|R_{\delta B}^{\mu}R_{\gamma A}^{\lambda}u(s)-u(s)\|^{2}ds<+\infty$.

\item[(ii)]$\lim\limits_{t\rightarrow+\infty}\|R_{\delta B}^{\mu}R_{\gamma A}^{\lambda}u(t)-u(t)\|=0$.

\item[(iii)]
There exists $\hat{u}\in \text{Fix}(R_{\delta B}^{\mu}R_{\gamma A}^{\lambda})$ such that
$w-\lim\limits_{t\rightarrow+\infty}u(t)=\hat{u}$.
\end{itemize}
\end{theorem}

\subsection{Rate of asymptotic regularity}

In this subsection, we are interested in the convergence rate of asymptotic regularity for the system (\ref{dynamic-system}).

\begin{theorem}\label{3th4.1}
Let $A : \mathbb{H}\rightrightarrows\mathbb{H}$ and $B : \mathbb{H}\rightrightarrows\mathbb{H}$ be
respectively maximally $\alpha$- and $\beta$-monotone.
Suppose that the parameters $\alpha,\beta,\gamma,\delta,\lambda,\mu$ satisfy  (\textbf{C2}),
$\text{zer}(A+B)\neq\emptyset$,
and that the assumption (\textbf{A1}) holds.
Let $u(t)$ be the trajectory of the system {\rm (\ref{dynamic-system})}.
If $\theta$ and $f$ are  subject to
 \begin{equation}\label{3th4.1-2}
\int^{+\infty}_{t_{0}}\theta(s)\int^{+\infty}_{s}\|f(\tau)\|d\tau ds<+\infty,
\end{equation}
then
\begin{equation}\label{3th4.1-2-0}
\|R_{\delta B}^{\mu}R_{\gamma A}^{\lambda}u(t)-u(t)\|=O\left(\frac{1}{\sqrt{t}}\right).
\end{equation}
In
particular,
if $f(t)\equiv 0$,
then
the convergence rate above can be improved to $o(\frac{1}{\sqrt{t}})$.
\end{theorem}

\begin{proof}
Observe that 
all the conditions in Theorem {\rm\ref{th3.2}} are satisfied.
Taking arbitrarily $u^{*}\in{\rm Fix}(R_{\delta B}^{\mu}R_{\gamma A}^{\lambda})$,
it is from the proof of Theorem \ref{th3.2} that
$\|R_{\delta B}^{\mu}R_{\gamma A}^{\lambda}u(t)-u(t)\|$ and $\|u(t)-u^{*}\|$
are bounded on $[t_{0},+\infty)$. To lighten the notion, let
\begin{equation}\label{3th4.1-s2-1}
\mathcal{M}_{\infty}=2\int^{+\infty}_{t_{0}}\|u(s)-u^{*}\|\|f(s)\|ds<+\infty
\end{equation}
and
$$\mathcal{M}_{RR}=\sup\limits_{t\in(t_{0},+\infty)}\|R_{\delta B}^{\mu}R_{\gamma A}^{\lambda}u(t)-u(t)\|<+\infty.$$
Consider the auxiliary function:
\begin{equation*}
V(t)=\|u(t)-u^{*}\|^{2}.
\end{equation*}
Then, with a parallel deducing of (\ref{ab-th3.1-s3}), we get
\begin{eqnarray}\label{3th4.1-s3}
\int^{t}_{t_{0}}\theta(s)\|R_{\delta B}^{\mu}R_{\gamma A}^{\lambda}u(s)-u(s)\|^{2}ds
&\leq& V(t_{0})+2\int^{t}_{t_{0}}\|u(s)-u^{*}\|\|f(s)\|ds\nonumber \\
&\leq&V(t_{0})+M_{\infty}.
\end{eqnarray}
On the other hand, by a parallel deducing of (\ref{3th3.2-s3}) and noting the nonexpansiveness of $R_{\delta B}^{\mu}R_{\gamma A}^{\lambda}$, one has
\begin{eqnarray*}
\frac{d}{d t}\left(\frac{1}{2}\|R_{\delta B}^{\mu}R_{\gamma A}^{\lambda}u(t)-u(t)\|^{2}\right)
\leq2\|f(t)\|\cdot\left\|R_{\delta B}^{\mu}R_{\gamma A}^{\lambda}u(t)-u(t)\right\|,
\end{eqnarray*}
which implies that for any $t_{0}\leq s\leq t$,
\begin{eqnarray}\label{3th4.1-s4}
\|R_{\delta B}^{\mu}R_{\gamma A}^{\lambda}u(t)-u(t)\|^{2}\leq\|R_{\delta B}^{\mu}R_{\gamma A}^{\lambda}u(s)-u(s)\|^{2}+2\mathcal{M}_{RR}\int^{t}_{s}\|f(\tau)\|d\tau.
\end{eqnarray}
Consequently, we have
\begin{eqnarray*}\label{3th4.1-s5}
&&\left(\inf\limits_{t\in[t_{0},+\infty)}\theta(t)\right)(t-t_{0})\|R_{\delta B}^{\mu}R_{\gamma A}^{\lambda}u(t)-u(t)\|^{2}\nonumber \\
&\leq&\|R_{\delta B}^{\mu}R_{\gamma A}^{\lambda}u(t)-u(t)\|^{2}\int^{t}_{t_{0}}\theta(s)ds\nonumber \\
&=&\int^{t}_{t_{0}}\theta(s)\|R_{\delta B}^{\mu}R_{\gamma A}^{\lambda}u(t)-u(t)\|^{2}ds\nonumber \\
&\overset{(\ref{3th4.1-s4})}{\leq}&\int^{t}_{t_{0}}\theta(s)\left(\|R_{\delta B}^{\mu}R_{\gamma A}^{\lambda}u(s)-u(s)\|^{2}
+2\mathcal{M}_{RR}\int^{t}_{s}\|f(\tau)\|d\tau\right)ds\nonumber \\
&\overset{(\ref{3th4.1-s3})}{\leq}&V(u_{0})+\mathcal{M}_{\infty}+2\mathcal{M}_{RR}
\int^{t}_{t_{0}}\theta(s)\int^{t}_{s}\|f(\tau)\|d\tau ds,
\end{eqnarray*}
namely,
\begin{eqnarray}\label{3th4.1-s6}
\|R_{\delta B}^{\mu}R_{\gamma A}^{\lambda}u(t)-u(t)\|^{2}\leq\frac{\mathcal{K}}{\left(\inf\limits_{t\in[t_{0},+\infty)}\theta(t)\right)(t-t_{0})},
\end{eqnarray}
where $\mathcal{K}=V(u_{0})+\mathcal{M}_{\infty}+2\mathcal{M}_{RR}\int^{+\infty}_{t_{0}}\theta(s)\int^{+\infty}_{s}\|f(\tau)\|d\tau ds<+\infty$
by (\ref{3th4.1-2}).
So, (\ref{3th4.1-2-0}) is verified.
When the perturbation $f$ is vanishing, noting that $R_{\delta B}^{\mu}R_{\gamma A}^{\lambda}$ is nonexpansive in case (\textbf{C2}),
we get the last conclusion by a parallel proof of \cite[Theorem 11]{c29Bot}.
The proof is complete.
\end{proof}

\begin{remark}
Let us mention that the condition (\ref{3th4.1-2}) is not restrictive.
In fact, by virtue of the Fubini's Theorem, (\ref{3th4.1-2}) is implied by the following condition:
\begin{equation}\label{remark-3th4.1}
\sup\limits_{t\in[t_{0},+\infty)}\theta(t)<+\infty\quad\text{and}\quad\int^{+\infty}_{t_{0}}s\|f(s)\|ds<+\infty.
\end{equation}
A straightforward example of (\ref{remark-3th4.1}) is that $\theta(t)\equiv M>0$, and $f(t)=\frac{1}{t^{p}}$ with $p>2$.
It is worth mentioning that a discretization version of (\ref{remark-3th4.1}) has been used  in \cite[Theorem 1]{Liang-2016} for a convergence rate analysis
of  the inexact Krasnosel'ski\u{i}-Mann iteration algorithm (\ref{3discret-ab1}) with $\mathcal{T}$ being a nonexpansive operator. Of course, the condition (\ref{3th4.1-2}) holds automatically when $f(t)\equiv 0$.
\end{remark}

\begin{remark}
Recently, Dao and Phan \cite{Dao-Phan2018} proved the global weak convergence of the adaptive DR algorithm (\ref{iterative2})  to a fixed point of the adaptive Douglas and Rachford operator $\tilde{T}$,  which is used to derive a solution of the problem (\ref{SMP}) in the "strongly+weakly" monotone setting. The $o(\frac{1}{\sqrt{k}})$ rate of asymptotic regularity of  $\tilde{T}$ was also established. See \cite[Theorem 4.5]{Dao-Phan2018}. As a comparison, we have established  continuous analogs for the adaptive Douglas-Rachford dynamical system (\ref{dynamic-system})  in [Theorem {\rm \ref{3th3.4}}, Theorem {\rm\ref{3th3.5-2}}, Theorem {\rm\ref{th3.2}} and Theorem \ref{3th4.1}].
\end{remark}

\subsection{ Exponential-type convergence }

In this subsection, we study the  exponential-type convergence   rate of the trajectory of  the system (\ref{dynamic-system}).
Note that $R_{\delta B}^{\mu}R_{\gamma A}^{\lambda}$ is quasi-nonexpansive in both cases (\textbf{C1}) and (\textbf{C2}).
The following  result on the exponential-type convergence  rate under the metric subregularity condition follows directly from Theorem \ref{3th4ad}.

\begin{theorem}\label{3th4.2}
Suppose that all the conditions in Theorem {\rm \ref{3th3.4}}, Theorem {\rm\ref{3th3.5-2}} or Theorem {\rm\ref{th3.2}} are satisfied.
Let $u(t)$ be the trajectory of  the system {\rm (\ref{dynamic-system})}, and let $\hat{u}\in{\rm Fix}(R_{\delta B}^{\mu}R_{\gamma A}^{\lambda})$ such that $w-\lim\limits_{t\rightarrow+\infty}u(t)=\hat{u}$.
Suppose that
$I-R_{\delta B}^{\mu}R_{\gamma A}^{\lambda}$ is metrically subregular at $\hat{u}$ for $0$ with a ball $B(\hat{u},r)$
and modulus $\kappa$.
If $r>\lim\limits_{t\rightarrow+\infty}\|u(t)-\hat{u}\|$,
then
there exist $t'\geq t_{0}$ and $M_{2}>0$ such that for all $t\geq t'$
  \begin{eqnarray*}
  &&{\rm dist}^{2}\left(u(t),{\rm Fix}(R_{\delta B}^{\mu}R_{\gamma A}^{\lambda})\right)\nonumber\\
  &\leq&
  e^{-\frac{1}{\kappa^{2}}\int^{t}_{t'}\theta(s)ds}
  \left(M_{2}\int^{t}_{t'}\|f(s)\|e^{\frac{1}{\kappa^{2}}\int^{s}_{t'}\theta(x)dx}ds+\|u(t')-\hat{u}\|^{2}\right).
  \end{eqnarray*}
Furthermore, if
\begin{equation}\label{3th4.2-1}
\int^{+\infty}_{t_0}\|f(s)\|e^{\frac{1}{\kappa^{2}}\int^{s}_{t_0}\theta(x)dx}ds<+\infty,
\end{equation}
then
  \begin{eqnarray*}
  {\rm dist}\left(u(t),{\rm Fix}(R_{\delta B}^{\mu}R_{\gamma A}^{\lambda})\right)
  =O\left(e^{-\frac{1}{\kappa^{2}}\int^{t}_{t'}\theta(s)ds}\right).
  \end{eqnarray*}
\end{theorem}

\begin{remark}
(i) Condition (\ref{3th4.2-1}) is  mild  and it is also satisfied when $f$ and $\theta$ are taken as in Remark  \ref{3remark-rate} (ii). Of course, it is satisfied automatically when $f(t)\equiv 0$. (ii) When  $A$ and $B$ are subdifferentials of  functions from $\mathbb{H}\to \mathbb{R}\cup\{+\infty\}$ and $\mu=\lambda=2$,  the metric subregularity of $I-R_{\delta B}^{\mu}R_{\gamma A}^{\lambda}$  has been used in  \cite[Theorem 6.1]{c29Guo} for establishing the local linear convergence of the DR algorithm for the "strongly+weakly" convex minimization problem.
\end{remark}

\begin{corollary}
Suppose that all the conditions in Theorem {\rm \ref{3th3.4}} or Theorem {\rm\ref{3th3.5-2}} are satisfied.
Let $u(t)$ be the trajectory of the system {\rm (\ref{dynamic-system})} with $f(t)\equiv 0$, and let $\hat{u}\in{\rm Fix}(R_{\delta B}^{\mu}R_{\gamma A}^{\lambda})$ such that $w-\lim\limits_{t\rightarrow+\infty}u(t)=\hat{u}$.
Suppose that
$I-R_{\delta B}^{\mu}R_{\gamma A}^{\lambda}$ is metrically subregular at $\hat{u}$ for $0$ with a ball $B(\hat{u},r)$
and modulus $\kappa$.
If $r>\lim\limits_{t\rightarrow+\infty}\|u(t)-\hat{u}\|$, then
there exist $t'\geq t_{0}$ and $\Theta_{0}>0$ such that
\begin{eqnarray*}
  {\rm dist}\left(u(t),{\rm Fix}(R_{\delta B}^{\mu}R_{\gamma A}^{\lambda})\right)
  \leq \|u(t')-\hat{u}\|e^{-\Theta_{0}\left(t-t'\right)},\quad t\geq t'.
\end{eqnarray*}
\end{corollary}

\begin{proof}
Set $$\Theta_{0}=-\frac{\inf\limits_{t\in[t',+\infty)}\theta(t)}{\kappa^{2}}.$$
Then $\Theta_{0}>0$ by the assumption (\textbf{A1}).
With the help of Corollary \ref{3co4ad}  we obtain the desired results.
The proof is complete.
\end{proof}

Next, we use a Lipschitz assumption instead of the metric subregularity
to derive another exponential-type convergence of the system (\ref{dynamic-system}) in case (\textbf{C2}).
Notice that a contraction operator (whose Lipschitz's constant less than 1) has a unique fixed point on $\mathbb{H}$ (by  the Banach-Picard's Theorem).

\begin{theorem}\label{3th4.6}
Let $A : \mathbb{H}\rightarrow\mathbb{H}$ be $\alpha$-monotone and Lipschitz continuous with constant $l$ such that $l\geq|\alpha|$,
and $B : \mathbb{H}\rightrightarrows\mathbb{H}$ be
maximally $\beta$-monotone.
Suppose that the parameters $\alpha,\beta,\gamma,\delta,\lambda,\mu$ satisfy (\textbf{C2}) and $\mu\neq 2+2 \gamma \alpha$.
Then $R_{\delta B}^{\mu}R_{\gamma A}^{\lambda}$ is Lipschitz continuous with constant
 \begin{eqnarray}\label{3th4.6-s1}\zeta
=\sqrt{1-\frac{\lambda(\mu-1)^{2}[(\lambda-1)(2+2 \gamma \alpha)-\lambda]}{1+2 \gamma \alpha+\gamma^{2} l^{2}}}<1.
\end{eqnarray}
Let $u(t)$ be the trajectory of the system {\rm (\ref{dynamic-system})} and ${\rm Fix}(R_{\delta B}^{\mu}R_{\gamma A}^{\lambda})=\{u^{*}\}$. Then there exists $M_{3}>0$ such that
  \begin{eqnarray}\label{3th4.3g-s2}
  &&\|u(t)-u^{*}\|^{2}\nonumber\\
  &\leq&
  e^{-(1-\zeta)\int^{t}_{t_{0}}\theta(s)ds}
  \left(M_{3}\int^{t}_{t_{0}}\|f(s)\|e^{(1-\zeta)\int^{s}_{t_{0}}\theta(x)dx}ds+\|u_{0}-u^{*}\|^{2}\right).
  \end{eqnarray}
Furthermore, if the assumption (\textbf{A2}) holds and $\int^{+\infty}_{t_{0}}\|f(s)\|e^{\int^{s}_{t_{0}}\theta(x)dx}ds<+\infty,$
then
  \begin{equation}\label{3th4.3g-s2-3}
  \|u(t)-u^{*}\|=O\left(e^{-\frac{1-\zeta}{2}\int^{t}_{t_{0}}\theta(s)ds}\right).
  \end{equation}
\end{theorem}

\begin{proof}
Recall that $R_{\gamma A}^{\lambda}=(1-\lambda) \mathrm{Id}+\lambda J_{\gamma A}$ and $R_{\delta B}^{\mu}=(1-\mu) \mathrm{Id}+\mu J_{\delta B}$.
Let us first show that
\begin{eqnarray}\label{3th4.3g-s3}
(\mu-1)(2+2\delta \beta)-\mu\geq0
\end{eqnarray}
and
\begin{eqnarray}\label{3th4.3g-s4}
(\lambda-1)(2+2 \gamma \alpha)-\lambda>0.
\end{eqnarray}
In fact, the condition (\ref{3th3.2-s1}) implies (\ref{para-sec-3}).
Noting that $(\lambda-1)(\mu-1)=1$, $\delta=(\lambda-1) \gamma$ and $\mu\neq 2+2 \gamma \alpha$, we have
\begin{eqnarray*}\label{3th4.3g-s5}
(\mu-1)(2+2\delta \beta)-\mu&=&2(\mu-1)+2 \gamma \beta-\mu\nonumber\\
&=&\mu-(2-2 \gamma\beta)
\geq0,
\end{eqnarray*}
and
\begin{eqnarray*}\label{3th4.3g-s6}
(\lambda-1)(2+2 \gamma \alpha)-\lambda
&=&(\lambda-1)\left((2+2 \gamma \alpha)-\frac{\lambda}{\lambda-1}\right)\nonumber\\
&=&\frac{\mu(2+2 \gamma \alpha-\mu)}{(\mu-1)^{2}\lambda}
>0.
\end{eqnarray*}
We then learn from Lemma \ref{prox-lip1} that for $x, y \in\mathbb{H}$,
\begin{eqnarray*}\label{3th4.3g-s7}
\|R_{\delta B}^{\mu} x-R_{\delta B}^{\mu} y\|^{2} &\leq& (\mu-1)^{2}\|x-y\|^{2}-\mu[(\mu-1)(2+2\delta \beta)-\mu]\|J_{\delta B}x-J_{\delta B}y\|^{2}\nonumber\\
&\overset{(\ref{3th4.3g-s3})}{\leq}&(\mu-1)^{2}\|x-y\|^{2},
\end{eqnarray*}
and
\begin{eqnarray*}\label{3th4.3g-s8}
\|R_{\gamma A}^{\lambda} x-R_{\gamma A}^{\lambda} y\|^{2} \leq \left((\lambda-1)^{2}
-\frac{\lambda[(\lambda-1)(2+2 \gamma \alpha)-\lambda]}{1+2 \gamma \alpha+\gamma^{2} l^{2}}\right)\|x-y\|^{2}.
\end{eqnarray*}
Hence,
\begin{eqnarray*}\label{3th4.3g-s9}
\|R_{\delta B}^{\mu}R_{\gamma A}^{\lambda} x-R_{\delta B}^{\mu}R_{\gamma A}^{\lambda} y\|^{2}&\leq& (\mu-1)^{2}\|R_{\gamma A}^{\lambda} x-R_{\gamma A}^{\lambda} y\|^{2}\nonumber\\
&\leq&\zeta\|x-y\|^{2},
\end{eqnarray*}
where
\begin{eqnarray*}\label{3th4.3g-s10}
\zeta
&=&(\mu-1)^{2}\left((\lambda-1)^{2}
-\frac{\lambda[(\lambda-1)(2+2 \gamma \alpha)-\lambda]}{1+2 \gamma \alpha+\gamma^{2} l^{2}}\right)\nonumber\\
&=&1-\frac{\lambda(\mu-1)^{2}[(\lambda-1)(2+2 \gamma \alpha)-\lambda]}{1+2 \gamma \alpha+\gamma^{2} l^{2}}\nonumber\\
&\overset{(\ref{3th4.3g-s4})}{<}& 1.
\end{eqnarray*}
Next, let us verify (\ref{3th4.3g-s2}).
Consider the following auxiliary function again:
\begin{equation}\label{3th4.3g-s11}
V(t)=\|u(t)-u^{*}\|^{2}.
\end{equation}
Similar to the derivation of (\ref{3th3ad3-s3}) with $\mathcal{T}=R_{\delta B}^{\mu}R_{\gamma A}^{\lambda}$,
we have
\begin{eqnarray}\label{3th4.3g-s12}
\frac{dV}{dt}&\leq&\theta(t)\|R_{\delta B}^{\mu}R_{\gamma A}^{\lambda}u(t)-R_{\delta B}^{\mu}R_{\gamma A}^{\lambda}(u^{*})\|^{2}-\theta(t)V(t)\nonumber \\
& &-\theta(t)\|R_{\delta B}^{\mu}R_{\gamma A}^{\lambda}u(t)-u(t)\|^{2}+2\sqrt{V(t)}\|f(t)\|\nonumber \\
&\leq&-(1-\zeta)\theta(t)V(t)+2\sqrt{V(t)}\|f(t)\|.
\end{eqnarray}
Note that all conditions in Theorem \ref{th3.2} are fulfilled and so $\sqrt{V(t)}=\|u(t)-u^{*}\|$ is bounded.
Multiplying (\ref{3th4.3g-s12}) by $e^{\int^{t}_{t_{0}}(1-\zeta)\theta(s)ds}$,
and then following the same roadmap of proof as that for Theorem \ref{3th4ad} we get (\ref{3th4.3g-s2}),
and so (\ref{3th4.3g-s2-3}) occurs whenever the assumption (\textbf{A2}) holds and $\int^{+\infty}_{t_{0}}\|f(s)\|e^{\int^{s}_{t_{0}}\theta(x)dx}ds<+\infty$.
The proof is complete.
\end{proof}

\begin{corollary}\label{Lip-exprate}
Let $A : \mathbb{H}\rightarrow\mathbb{H}$ be $\alpha$-monotone and Lipschitz continuous with constant $l$ such that $l\geq|\alpha|$,
and $B : \mathbb{H}\rightrightarrows\mathbb{H}$ be
maximally $\beta$-monotone.
Suppose that the parameters $\alpha,\beta,\gamma,\delta,\lambda,\mu$ satisfy case (\textbf{C2}), $\mu\neq 2+2 \gamma \alpha$,
and that $\theta(t)$ satisfies the assumption (\textbf{A1}).
Let $u(t)$ be the unique solution of the system {\rm (\ref{dynamic-system})} with $f(t)\equiv0$,
and let ${\rm Fix}(R_{\delta B}^{\mu}R_{\gamma A}^{\lambda})=\{u^{*}\}$.
Then there exists $\Theta_{1}>0$ such that
  \begin{eqnarray*}
  \|u(t)-u^{*}\|\leq\|u_{0}-u^{*}\|e^{-\Theta_{1}(t-t_{0})}.
  \end{eqnarray*}
\end{corollary}

\begin{proof}
It follows from Theorem \ref{3th4.6} that $R_{\delta B}^{\mu}R_{\gamma A}^{\lambda}$ is Lipschitz continuous with a constant $\zeta\in(0,1)$.
Set $$\Theta_{1}=\frac{1-\zeta}{2}\inf\limits_{t\in[t_{0},+\infty)}\theta(t).$$
Then $\Theta_{1}>0$ by the assumption (\textbf{A1}).
With the help of $f(t)\equiv0$ and (\ref{3th4.3g-s2}), the conclusion follows.
The proof is complete.
\end{proof}

\begin{remark}
Corollary {\rm \ref{Lip-exprate}} can be viewed as a continuous analog to \cite[Theorem 4.8]{Dao-Phan2018} in which a global linear convergence rate of the adaptive DR algorithm (\ref{iterative2}) was established  under  Lipschitz assumption.
\end{remark}

In next theorem we take into account the condition $f\in\mathbb{L}^{1}([t_{0},b])$ for all $b\geq t_{0}$ instead of $f\in\mathbb{L}^{1}([t_{0},+\infty))$,
and impose a constraint on $\theta$.
Note that, with such modifications, Theorem \ref{th3.1} still holds according to the Cauchy-Lipschitz theorem.

\begin{theorem}\label{3th4.7}
Let $A : \mathbb{H}\rightarrow\mathbb{H}$ be $\alpha$-monotone and Lipschitz continuous with constant $l$ which satisfies $l\geq|\alpha|$,
and $B : \mathbb{H}\rightrightarrows\mathbb{H}$ be
maximally $\beta$-monotone.
Suppose that the parameters $\alpha,\beta,\gamma,\delta,\lambda,\mu$ satisfy case (\textbf{C2}) and $\mu\neq 2+2 \gamma \alpha$.
Let $u(t)$ be the trajectory of the system {\rm (\ref{dynamic-system})}, and ${\rm Fix}(R_{\delta B}^{\mu}R_{\gamma A}^{\lambda})=\{u^{*}\}$.
Suppose that $(1-\zeta)\theta(t)>1$ for any $t\in[t_{0},+\infty)$,
and that $f\in\mathbb{L}^{1}([t_{0},b])$ for all $b\geq t_{0}$.
Then
\begin{eqnarray*}\label{3th4.3g-s2-4}
  &&\|u(t)-u^{*}\|^{2}\nonumber\\
  &\leq&
  e^{-\int^{t}_{t_{0}}(1-\zeta)\theta(s)-1ds}\left(
  \int^{t}_{t_{0}}\|f(s)\|^{2}e^{\int^{s}_{t_{0}}(1-\zeta)\theta(x)-1dx}ds+\|u_{0}-u^{*}\|^{2}\right),
\end{eqnarray*}
where $\zeta$ is defined in {\rm(\ref{3th4.6-s1})}.
Furthermore, if $\int^{+\infty}_{t_{0}}[(1-\zeta)\theta(s)-1]ds=+\infty$ and\\ $\int^{+\infty}_{t_{0}}\|f(s)\|^2e^{\int^{s}_{t_{0}}(1-\zeta)\theta(x)-1dx}ds<+\infty,$
  then
  \begin{equation*}
  \|u(t)-u^{*}\|=O\left(e^{-\frac{1}{2}\int^{t}_{t_{0}}(1-\zeta)\theta(s)-1ds}\right).
  \end{equation*}
\end{theorem}

\begin{proof}
It follows from Theorem \ref{3th4.6} (i) that
$R_{\delta B}^{\mu}R_{\gamma A}^{\lambda}$ is Lipschitz continuous with $\zeta$ defined in (\ref{3th4.6-s1}),
and has a unique fixed point $\{u^{*}\}={\rm Fix}(R_{\delta B}^{\mu}R_{\gamma A}^{\lambda})$.
Consider the following auxiliary function again:
\begin{equation*}
V(t)=\|u(t)-u^{*}\|^{2}.
\end{equation*}
  It turns out from (\ref{3th4.3g-s11}) and (\ref{3th4.3g-s12}) that
  \begin{eqnarray*}\label{3th4.3g-s13}
  \frac{dV}{dt}&\leq&-\frac{\theta(t)}{\kappa^{2}}V(t)+2\sqrt{V(t)}\|f(t)\|\nonumber\\
  &\leq&-[(1-\zeta)\theta(t)-1]V(t)+\|f(t)\|^{2}.
  \end{eqnarray*}
  Similar to the proofs of (\ref{3th4.3g-s2}) and (\ref{3th4.3g-s2-3}), noticing that $(1-\zeta)\theta(t)>1$,
  we obtain the desired results. The proof is complete.
\end{proof}

\begin{remark}The functions
$$\theta_{1}(t)=\frac{2t+1}{2t(1-\zeta)}\quad\text{and}\quad f_{1}(t)=\left(\frac{1}{t}e^{-\int^{t}_{t_{0}}\frac{1}{s}ds}\right)^{\frac{1}{2}}=\frac{\sqrt{t_{0}}}{t},\quad t\geq t_{0}>0,$$
verify the conditions $(1-\zeta)\theta_{1}(t)-1>0$, $f_{1}\in\mathbb{L}^{1}([t_{0},b])$ for all $b\geq t_{0}$, and\\ $\int^{+\infty}_{t_{0}}\|f_{1}(s)\|^{2}e^{\int^{s}_{t_{0}}(1-\zeta)\theta_{1}(x)-1dx}ds<+\infty$,
while $\int^{+\infty}_{t_{0}}\|f_{1}(s)\|e^{(1-\zeta)\int^{s}_{t_{0}}\theta_{1}(x)dx}ds<+\infty$ and $f_{1}\in\mathbb{L}^{1}([t_{0},+\infty))$ fail.
On the other hand, the functions
$$\theta_{2}(t)=\frac{1}{1-\zeta}\quad\text{and}\quad f_{2}(t)=\frac{1}{t^{2}}e^{-(t-t_{0})},$$
cater for the conditions $\int^{+\infty}_{t_{0}}\|f_{2}(s)\|e^{(1-\zeta)\int^{s}_{t_{0}}\theta_{2}(x)dx}ds<+\infty$, $f_{2}\in\mathbb{L}^{1}([t_{0},+\infty))$ and the assumption (\textbf{A2}),
while $(1-\zeta)\theta_{2}(t)-1>0$ and $\int^{t}_{t_{0}}(1-\zeta)\theta_{2}(s)-1ds=+\infty$ are not fulfilled.
This indicates that the assumptions on $\theta$ and $f$ in Theorem \ref{3th4.6} are independent of those in Theorem \ref{3th4.7}.
\end{remark}

\begin{remark} As pointed out in \cite[Remark 4.15]{Dao-Phan2018} that the sum of $\alpha$- and $\beta$-monotone operators $A$ and $B$ with $\alpha+\beta\ge 0$ can be transformed into the sum of two new monotone operators $\tilde A$ and $\tilde B$ with
$$\tilde A=A+\frac{\beta-\alpha}{2}I \text{ and } \tilde B=B+\frac{\beta-\alpha}{2}I.$$
Then one can propose dynamical systems for the problem associated with monotone operators $\tilde A$ and $\tilde B$.  Here, our main goal is to design  the dynamical system approach in which one operates only original data. This  might be especially helpful when the resolvents are given as black boxes, in which case one just needs to adjust the approach using corresponding parameters.
\end{remark}

\section{Applications to structural minimization problems}\label{3aplication}
In this section, we apply the results obtained for the  adaptive Douglas-Rachford dynamical system (\ref{dynamic-system}) to solve the "strongly+weakly" convex minimization problem  \cite{K-Guo2018,c29Guo,ZM-OP}:
\begin{equation}
   \label{SMP2}
    \min\limits_{x\in \mathbb{H}}\phi(x)+\varphi(x),
\end{equation}
where $\phi$, $\varphi$ : $\mathbb{H}\rightarrow \mathbb{R}\cup \{+\infty\}$
be two proper and closed functions,
one of which is strongly convex while the other one is weakly convex. Let us recall some necessary  concepts and results in convex analysis. Let $\text{dom}\,h$ denote the effective domain of
a proper function $h:\mathbb{H}\to \mathbb{R}\cup +\infty$, i.e., $\text{dom}\,h:=\{x\in\mathbb{H}\mid h(x)<+\infty\}\ne\emptyset$.
Let $\operatorname{Prox}_{\tau h}:\mathbb{H}\rightrightarrows\mathbb{H}$ denote the proximity operator of $h$, i.e.,
$$
\operatorname{Prox}_{\tau h}(x) :=\underset{z \in \mathbb{H}}{\operatorname{argmin}}\left(h(z)+\frac{1}{2 \tau}\|z-x\|^{2}\right),\quad \forall x\in\mathbb{H}.
$$
The function $h$ is said to be $\alpha$-convex (see, e.g., \cite[Definition 4.1]{Vial1893}) for some $\alpha\in\mathbb{R}$,
if $\forall x,y\in\text{dom}\,h$, $\forall \tau\in(0,1)$,
$$
h((1-\tau) x+\tau y)+\frac{\alpha}{2} \tau(1-\tau)\|x-y\|^{2} \leq(1-\tau) h(x)+\tau h(y).
$$
We say that $h$ is convex, strongly convex and weakly convex if $\alpha= 0$, $\alpha > 0$ and  $\alpha < 0$, respectively.
We use $\hat{\partial} h$ denote the Fr\'{e}chet subdifferential of $h$, which is defined by
$$
\hat{\partial} h(x)=\left\{z\in \mathbb{H}\mid\liminf\limits_{y \rightarrow x} \frac{h(y)-h(x)-\left\langle z, y-x\right\rangle}{\|y-x\|} \geq 0\right\}.
$$
Notice that
if $h$ is convex, then
$$
\hat{\partial} h(x)=\partial h(x):=\left\{v \in \mathcal{H}\mid h(y) \geq h(x)+\langle v, y-x \rangle \quad \forall y\in \text{dom}\,h \right\},
$$
see, e.g., \cite[Theorem 1.93]{{Mordukhovich}}. The following lemma comes from \cite[Lemma 5.2]{Dao-Phan2018}.
\begin{lemma}\label{3le5.1}
Let $h : \mathbb{H}\rightarrow\mathbb{R}\cup \{+\infty\}$ be proper, closed, and
$\alpha-$convex. Suppose $\gamma\in\mathbb{R}_{++}$ and $1+\gamma\alpha>0$. Then the following conclusions hold:
\begin{itemize}
  \item [(i)]$\hat{\partial} h$ is maximally $\alpha$-monotone.
  \item [(ii)]$\operatorname{Prox}_{\gamma h}=J_{\gamma \hat{\partial} h}$ is single-valued and has full domain.
\end{itemize}
\end{lemma}

We are now in position to deal with the "strongly+weakly" convex minimization problem (\ref{SMP2}). To do so, we propose
the following dynamical system:
\begin{equation}\label{dynamic-system5}
 \left\{
             \begin{array}{lr}
             \frac{du}{dt}+\theta(t)\left[u(t)-\mathrm{R}_{2}\mathrm{R}_{1}(u(t))\right]=f(t), &  \\
             u(t_{0})=u_{0}\in\mathbb{H},&
             \end{array}
   \right.
\end{equation}
where
\begin{equation*}
\begin{aligned}
&\mathrm{R}_{1} :=(1-\lambda) \mathrm{Id}+\lambda P_{1},\quad
\mathrm{R}_{2} :=(1-\mu) \mathrm{Id}+\mu P_{2},&\\
&P_{1}:=\operatorname{Prox}_{\gamma \phi} \quad\text{and}\quad P_{2}:=\operatorname{Prox}_{\delta \varphi}.&
\end{aligned}
\end{equation*}
Clearly, the system (\ref{dynamic-system5}) is a special case of the adaptive Douglas-Rachford dynamical system  (\ref{dynamic-system}) (by seting $A=\hat{\partial} \phi$ and $B=\hat{\partial} \varphi$). We can learn from Theorem \ref{th3.1} that for each initial point $u_{0}$,
there exists a unique absolutely continuous trajectory $u(t)$
of the system {\rm(\ref{dynamic-system5})} in the global time interval
$[t_{0},+\infty)$.
Based on such a fact, Lemma \ref{3le5.1} allows us to get the parallel results from the previous sections.
Note that ${\rm Fix}(\mathrm{R}_{2}\mathrm{R}_{1})=zer(\hat{\partial} \phi+\hat{\partial} \varphi)\subseteq\arg\min(\phi+\varphi)$
by \cite[Lemma 5.3]{Dao-Phan2018}, and that $\phi+\varphi$ is strongly convex in the case (\textbf{C1}) due to $\alpha+\beta>0$, which leads to the fact that
the problem (\ref{SMP2}) has a unique solution.

\begin{theorem}\label{3th5.1}
Let $\phi : \mathbb{H}\rightarrow\mathbb{R}\cup \{+\infty\}$ and
$\varphi : \mathbb{H}\rightarrow\mathbb{R}\cup \{+\infty\}$ be proper, closed, and
respectively  $\alpha$- and $\beta$-convex.
Suppose that the parameters $\alpha,\beta,\gamma,\delta,\lambda,\mu$ satisfy (\textbf{C1}),
and that the assumption (\textbf{A1}) holds.
Let $u(t)$ be the trajectory of the system {\rm (\ref{dynamic-system5})}.
Then, for any $u^{*}\in{\rm Fix}(\mathrm{R}_{2}\mathrm{R}_{1})$, the following
statements are true:
\begin{itemize}
\item[(i)]$\int^{+\infty}_{t_{0}}\|\mathrm{R}_{2}\mathrm{R}_{1}u(s)-u(s)\|^{2}ds<+\infty$.

\item[(ii)]$\lim\limits_{t\rightarrow+\infty}\|\mathrm{R}_{2}\mathrm{R}_{1}u(t)-u(t)\|=0$.

\item[(iii)] If $\mathrm{R}_{2}\mathrm{R}_{1}$ satisfies the demiclosedness principle, then
there exists $\hat{u}\in \text{Fix}(\mathrm{R}_{2}\mathrm{R}_{1})$ such that
$w-\lim\limits_{t\rightarrow+\infty}u(t)=\hat{u}$.

\item[(iv)]
$\int^{+\infty}_{t_{0}}\left\|\alpha\left(P_{1}u(s)-P_{1} u^{*}\right)
+\beta\left(P_{2} \mathrm{R}_{1} u(s)-P_{2}\mathrm{R}_{1}u^{*} \right)\right\|^{2}ds<+\infty$.
\item[(v)]
$\lim\limits_{t\rightarrow+\infty}\left\|\alpha\left(P_{1}u(t)-P_{1} u^{*}\right)
+\beta\left(P_{2} \mathrm{R}_{1} u(t)-P_{2} \mathrm{R}_{1}u^{*} \right)\right\|=0$.
\item[(vi)]
$\lim\limits_{t\rightarrow+\infty}P_{1}u(t)=P_{1} u^{*}=\lim\limits_{t\rightarrow+\infty}P_{2} \mathrm{R}_{1} u(t)=P_{2} \mathrm{R}_{1}u^{*}=\text{zer}(A+B)$.

\end{itemize}
\end{theorem}

\begin{remark}

Some  convergence results corresponding to the case $\theta(t)\equiv\bar{\Theta}>0$, $f(t)\equiv0$ and $\mathbb{H}=\mathbb{R}^{n}$ have been investigated in  \cite[Theorem 3.3 and Theorem 3.7]{ZM-OP}, which are covered by Theorem \ref{3th5.1}.
\end{remark}

\begin{theorem}\label{3th5.2}
Let $\phi : \mathbb{H}\rightarrow\mathbb{R}\cup \{+\infty\}$ and
$\varphi : \mathbb{H}\rightarrow\mathbb{R}\cup \{+\infty\}$ be proper, closed, and
respectively $\alpha$- and $\beta$-convex.
Suppose that the parameters $\alpha,\beta,\gamma,\delta,\lambda,\mu$ satisfy (\textbf{C2}),
$zer(\hat{\partial} \phi+\hat{\partial} \varphi)\neq\emptyset$,
and that the the assumption (\textbf{A1}) holds.
Let $u(t)$ be the trajectory of the system {\rm (\ref{dynamic-system5})}.
Then, for any $u^{*}\in{\rm Fix}(\mathrm{R}_{2}\mathrm{R}_{1})$, the following
statements are true:
\begin{itemize}
\item[(i)]
$\int^{+\infty}_{t_{0}}\|\mathrm{R}_{2}\mathrm{R}_{1}u(s)-u(s)\|^{2}ds<+\infty$.

\item[(ii)]$\lim\limits_{t\rightarrow+\infty}\|\mathrm{R}_{2}\mathrm{R}_{1}u(t)-u(t)\|=0$.

\item[(iii)]
There exists $\hat{u}\in \text{Fix}(\mathrm{R}_{2}\mathrm{R}_{1})$ such that
$w-\lim\limits_{t\rightarrow+\infty}u(t)=\hat{u}$.

\item[(iv)]If $\alpha+\beta>0$, then\\
$\int^{+\infty}_{t_{0}}\left\|w_{1}\left(P_{1}u(s)-P_{1} u^{*}\right)
+w_{2}\left(P_{2} \mathrm{R}_{1} u(s)-P_{2} \mathrm{R}_{1}u^{*} \right)\right\|^{2}ds<+\infty$,\\
$\lim\limits_{t\rightarrow+\infty}\left\|w_{1}\left(P_{1}u(t)-P_{1} u^{*}\right)
+w_{2}\left(P_{2} \mathrm{R}_{1} u(t)-P_{2} \mathrm{R}_{1}u^{*} \right)\right\|=0$, and\\
$\lim\limits_{t\rightarrow+\infty}P_{1}u(t)=P_{1} u^{*}=\lim\limits_{t\rightarrow+\infty}P_{2} \mathrm{R}_{1} u(t)=P_{2} \mathrm{R}_{1}u^{*}=\text{zer}(A+B)$.

\end{itemize}
\end{theorem}




\begin{remark}\label{apli-remark}
Guo et al. \cite{c29Guo} proposed  a DR algorithm solving the problem (\ref{SMP2}) in an Euclidean space and discussed its convergence.
The algorithm proposed by Guo et al. \cite{c29Guo} can be regard as a special case of a discretization version of the system (\ref{dynamic-system5}).
The convergence results in \cite{c29Guo} require that the strong convexity of the objective function strictly outweighs the weak counterpart, that is, $\alpha+\beta>0$.
Convergence of the same DR algorithm, only for the case $\alpha+\beta=0$, has also been considered in \cite{K-Guo2018} under the condition that one function is strongly convex with Lipschitz continuous gradient.
In contrast, we assume $\alpha+\beta\geq0$, and the convergence is still guaranteed without any differentiability assumption; see Theorem \ref{3th5.1} and Theorem \ref{3th5.2}.
\end{remark}

\begin{theorem}\label{3th5.3}
Let $\phi : \mathbb{H}\rightarrow\mathbb{R}\cup \{+\infty\}$ and
$\varphi : \mathbb{H}\rightarrow\mathbb{R}\cup \{+\infty\}$ be proper, closed, and
respectively $\alpha$- and $\beta$-convex.
Suppose that the parameters $\alpha,\beta,\gamma,\delta,\lambda,\mu$ satisfy (\textbf{C2}),
$zer(\hat{\partial} \phi+\hat{\partial} \varphi)\neq\emptyset$,
and that the assumption (\textbf{A1}) holds.
Let $u(t)$ be the trajectory of the system {\rm (\ref{dynamic-system5})}.
If $\theta$ and $f$ are subject to
\begin{equation*}\label{3th5.3-2}
 \int^{+\infty}_{t_{0}}\theta(s)\int^{+\infty}_{s}\|f(\tau)\|d\tau ds<+\infty,
\end{equation*}
then
$\|\mathrm{R}_{2}\mathrm{R}_{1}u(t)-u(t)\|=O\left(\frac{1}{\sqrt{t}}\right).$
In particular,
if $f(t)\equiv0$,
then the convergence rate above can be improved to $o(\frac{1}{\sqrt{t}})$.
\end{theorem}

\begin{theorem}\label{3th5.4}
Suppose that all the conditions in Theorem \ref{3th5.1} or Theorem \ref{3th5.2} are satisfied.
Let $\hat{u}\in{\rm Fix}(\mathrm{R}_{2}\mathrm{R}_{1})$ such that $w-\lim\limits_{t\rightarrow+\infty}u(t)=\hat{u}$.
Suppose that
$I-\mathrm{R}_{2}\mathrm{R}_{1}$ is metrically subregular at $\hat{u}$ for $0$ with a ball $B(\hat{u},r)$
and modulus $\kappa$.
If $r>\lim\limits_{t\rightarrow+\infty}\|u(t)-\hat{u}\|$,
then
there exist $t'\geq t_{0}$ and $M_{4}>0$ such that for all $t\geq t'$
\begin{eqnarray*}
 &&{\rm dist}^{2}(u(t),{\rm Fix}(\mathrm{R}_{2}\mathrm{R}_{1}))\nonumber\\
  &\leq&
  e^{-\frac{1}{\kappa^{2}}\int^{t}_{t'}\theta(s)ds}
  \left(M_{4}\int^{t}_{t'}\|f(s)\|e^{\frac{1}{\kappa^{2}}\int^{s}_{t'}\theta(x)dx}ds+\|u(t')-\hat{u}\|^{2}\right).
  \end{eqnarray*}
  Furthermore,
  if $\int^{+\infty}_{t_0}\|f(s)\|e^{\frac{1}{\kappa^{2}}\int^{s}_{t_0}\theta(x)dx}ds<+\infty$, then
  $${\rm dist}\left(u(t),{\rm Fix}(\mathrm{R}_{2}\mathrm{R}_{1}))\right)
  =O\left(e^{-\frac{1}{\kappa^{2}}\int^{t}_{t'}\theta(s)ds}\right).$$
\end{theorem}

\begin{remark}

Theorem \ref{3th5.4} covers  \cite[Theorem 4.1 and Theorem 4.3]{ZM-OP} where $\theta(t)\equiv\bar{\Theta}>0$, $f(t)\equiv0$ and $\mathbb{H}=\mathbb{R}^{n}$.
It is worth mentioning that a local linear convergence rate of the DR algorithm for the problem (\ref{SMP2}) was derived in \cite{c29Guo} under the metric subregularity condition.
\end{remark}

\section{Conclusion}

In this paper we studied the asymptotic behavior of the nonautonomous evolution equation governed by a quasi-nonexpansive operator in Hilbert spaces. We proved the weak convergence of the trajectory to a fixed  point of the operator by relying on  Lyapunov analysis, and then established a flexible global exponential-type rate under the metric subregularity condition. We also analyzed the convergence of the trajectories of the adaptive Douglas-Rachford dynamical system which is applied for finding a zero of  "strongly+weakly" monotone operators. We derived  continuous time analogs to the corresponding results of Liang et al. \cite{Liang-2016} and Dao and Phan \cite{Dao-Phan2018}.



\begin{thebibliography}{99}

\bibitem{application2Bruckstein}   A.M. Bruckstein, D.L. Donoho,  M. Elad,   From sparse solutions of systems of equations to sparse
modeling of signals and images, SIAM Rev. 51 (2009)  34--81.

\bibitem{b9Douglas}  J. Douglas  H.H. Rachford, On the numerical solution of heat conduction problems in two or three
space variables, Trans. Amer. Math. Soc. 82 (1956)
 pp. 421--439.

\bibitem{b11Lions}   P.L. Lions, B. Mercier, Splitting algorithms for the sum of two nonlinear operators, SIAM J. Numer. Anal. 16 (1979)  964--979.


\bibitem{Svaiter2011}  B.F. Svaiter, On weak convergence of the Douglas-Rachford method,  SIAM J. Control Optim. 49(1) (2011)   280--287.

\bibitem{b10Eckstein}   J. Eckstein, D.P. Bertsekas, On the Douglas-Rachford splitting method and the proximal point algorithm
for maximal monotone operators,  Math. Program. 55 (1992)  293--318.

\bibitem{Lawrence1987}   J. Lawrence J.E. Spingarn, On fixed points of nonexpansive piecewise isometric mappings,
Proc. London Math. Soc. 55(3) (1987)  605--624.

\bibitem{Gabay1983}  D. Gabay, Application of the method of multipliers to varuational inequalities,
In: Fortin, M., Glowinski, R. (eds.) Augmented Lagrangian Methods: Application to the Numerical Solution of Boundary-Value Problem,
pp. 299-331. North-Holland, Amsterdam, 1983.

\bibitem{Dao-Phan2018jg}  M.N. Dao, H.M. Phan,  Linear convergence of the generalized Douglas-Rachford algorithm for feasibility problems, J. Global Optim. 72 (3) (2018)  443--474.

\bibitem{b12-1He}   B.S. He, X.M. Yuan,  On the convergence rate of the Douglas-Rachford operator splitting
method,  Math. Program. 153 (2015)  715--722.

\bibitem{HeB2012}    B.S. He, X.M. Yuan, On the O(1/n) convergence rate of the douglas-rachford alternating direction method,
SIAM J. Numer. Anal. 50(2) (2012)
  700--709.

\bibitem{Phan2016}   H.M. Phan, Linear convergence of the Douglas-Rachford method for two closed sets, Optimization 65(2) (2016)  369--385.

\bibitem{b12Bayram}  I. Bayram, I.W. Selesnick, The Douglas-Rachford algorithm for weakly convex penalties, Preprint arXiv:1511.03920v1 (2015)

\bibitem{K-Guo2018}  K. Guo, D.R. Han, A note on the Douglas-Rachford splitting method for optimization problems
involving hypoconvex functions, J. Glob. Optim. 72(3) (2018) 431--441.

\bibitem{c29Guo}  K. Guo, D.R. Han, X.M. Yuan, Convergence analysis of Douglas-Rachford splitting method for ``strongly+weakly" convex programming,
SIAM J. Numer. Anal. 55 (2017)  1549--1577.

\bibitem{Dao-Phan2018}  M.N Dao, H.M. Phan, Adaptive Douglas-Rachford splitting algorithm for the sum of two operators, SIAM J. Optim. 29(4) (2019)  2697--2724.

\bibitem{lipschitz}  H.H. Bauschke, P.L. Combettes, Convex Analysis and Monotone Operator Theory in
Hilbert Spaces, Springer, Cham, 2017.

\bibitem{Cegielski2015}   Cegielski, Application of quasi-nonexpansive operators to an iterative method for
variational inequality, SIAM J. Optim. 25, (2015)  2165--2181.

\bibitem{Cegielski2018}  A. Ciegielski, S. Reich, R. Zalas, Regular sequences of quasi-nonexpansive operators and their applications, SIAM J. Optim. 28 (2018)  1508--1532.

\bibitem{Krasn}  M.A. Krasnosel'ski\u{i}, Two remarks on the method of successive approximations,  Uspekhi Mat. Nauk. 63 (1955)  123--127.

\bibitem{Kolobov}  V.I. Kolobov, S.R.and R. Zalas, Weak, strong, and linear convergence of a double-layer fixed point algorithm, SIAM J. Optim. 27(3) (2017)  1431--1458.

\bibitem{Mann}  W.R. Mann, Mean value methods in iteration,  Proc. Am. Math. Soc. 4 (1955)   506--510.

\bibitem{Baillon}   J.B. Baillon, R.E. Bruck, The rate of asymptotic regularity is $O(\frac{1}{\sqrt{n}})$,  Lect. Notes Pure Appl. Math. 178 (1996)  51--81.

\bibitem{Cominetti}  R. Cominetti, J.A. Soto, J. Vaisman, On the rate of convergence of Krasnoselski-Mann iterations
and their connection with sums of Bernoullis, Israel J. Math. 199(2) (2014)  757--772.

\bibitem{Liang-2016}   J. Liang, J. Fadili, G. Peyre,: Convergence rates with inexact non-expansive operators, Math. Program.
159(1-2) (2016)  403--434.

\bibitem{Bravo-2019}  M. Bravo, R. Cominetti,, M. Pavez-Sign\'{e}, Rates of convergence for inexact Krasnosel'ski\u{i}-Mann
iterations in Banach spaces,  Math. Program. 175(1-2) (2019)  241--262.

\bibitem{c29Bot}   R.I. Bot, E.R. Csetnek, A dynamical system associated with the fixed points set of a nonexpansive
operator,  J. Dynam. Differential Equations 29(1) (2017)  155--168.

\bibitem{Arrow1957}  K. Arrow, L. Hurwicz, Gradient methods for constrained maxima,  Oper. Res. 5(2) (1957)  258--265.


\bibitem{c23Abbas}  B. Abbas,  H. Attouch, B.F.Svaiter, Newton-like dynamics and forward-backward methods for
structured monotone inclusions in Hilbert spaces, J. Optim. Theory Appl. 161 (2014)  331--360.

\bibitem{c24Abbas}    B. Abbas, H. Attouch, Dynamical systems and forward-backward algorithms associated
with the sum of a convex subdifferential and a monotone cocoercive operator,  Optimization 64 64 (2015)  2223--2252.


\bibitem{c22Attouch}  H. Attouch, B.F. Svaiter, A continuous dynamical Newton-like approach to solving monotone
inclusions, SIAM J. Control Optim. 49 (2011)   574--598.

\bibitem{c26Bolte}  J. Bolte, Continuous gradient projection method in Hilbert spaces,  J. Optim. Theory Appl. 119 (2003), 235--259.

\bibitem{c30Bot}  R.I. Bot, E.R. Csetnek, Convergence rates for forward-backward dynamical systems associated with
strongly monotone inclusions, J. Math. Anal. Appl.  457(2) (2018)  1135--1152.

\bibitem{c31Bot}  R.I. Bot, E.R. Csetnek, Second order forward-backward dynamical systems for monotone inclusion
problems,  SIAM J. Control  Optim. 54(3)  (2016)  1423--1443.

\bibitem{2019Csetnek}  E.R. Csetnek, Y. Malitsky, M.K. Tam, Shadow Douglas-Rachford splitting for monotone inclusions,
Appl. Math. Optim., (2019). doi:10.1007/s00245-019-09597-8.

\bibitem{ZM-OP} M. Zhu, R. Hu, Y. Fang, A continuous dynamical splitting method for solving
`strongly+weakly' convex programming problems. Optimization 69 (6) (2020) 1335--1359.

\bibitem{stronglyconvex}   R.T. Rockafellar, R.J.B. Wets, Variational Analysis, Springer, Berlin, 2010.



















\bibitem{Borwein2010}  J.M. Borwein, Fifty years of maximal monotonicity,  Optim. Lett 4(4) (2010)  473--490.

\bibitem{1972Brezis}   H. Br\'{e}zis, Op\'{e}rateurs maximaux monotones dans les espaces de Hilbert et \'{e}quations d'\'{e}volution,
Lecture Notes 5, North Holland, 1972.

\bibitem{Haraux2015}  A. Haraux, M.A. Jendoubi, The Convergence Problem for Dissipative Autonomous Systems: Classical Methods and Recent Advances,
Springer, Heidelberg, 2015.

\bibitem{Metric-subregularity}   A.L. Dontchev, R.T. Rockafellar, Implicit Functions and Solution Mappings, Springer, New York, 2009.

\bibitem{Wang2011}   F. Wang, H.K. Xu, Cyclic algorithms for split feasibility problems in Hilbert spaces, Nonlinear Anal. 74 (2011)  4105--4111.

\bibitem{CegielskiJOTA}  A. Cegielski, General method for solving the split common fixed point problem, J. Optim.
Theory Appl. 165 (2015)  385--404.




\bibitem{Teschl-stability}   G. Teschl,  Ordinary Differential Equations and Dynamical Systems, American Mathematical Society, Providence, 2012.


\bibitem{Vial1893}   J.P. Vial, Strong and weak convexity of sets and functions,  Math. Oper. Res. 8(2) (1983)  231--259.

\bibitem{Mordukhovich}  B.S. Mordukhovich,  Variational Analysis and Generalized Differentiation I: Basic Theory, Springer, Berlin, 2006.




\end{thebibliography}
\end{document}